\documentclass[12pt]{article}

\usepackage{amsmath, latexsym, amsfonts, amssymb, amsthm, amscd, enumerate, comment, stackrel, mathrsfs}
\usepackage[square,sort,comma,numbers]{natbib}

\usepackage[normalem]{ulem}
\usepackage[usenames,dvipsnames,svgnames,table]{xcolor}
\usepackage[left=3cm, right=2.5cm, top=2.5cm, bottom=2.5cm]{geometry}
\usepackage[colorlinks=true, linkcolor=blue, urlcolor=blue, citecolor=blue]{hyperref}

\newtheorem{theorem}{Theorem}[section]

\newtheorem{lemma}[theorem]{Lemma}

\theoremstyle{definition}
\newtheorem{remark}[theorem]{Remark}
\newtheorem{definition}[theorem]{Definition}

\numberwithin{equation}{section}

\newcommand{\N}{\mathbb{N}}
\newcommand{\cF}{\mathcal{F}}
\newcommand{\E}{\mathbb{E}}
\newcommand{\kzero}{7}

\newcommand{\nat}[1]{{\color{black} #1}} 

\title{The contact process with fitness on Galton-Watson trees}
\author{Natalia Cardona-Tob\'on  and Marcel Ortgiese }

\begin{document}

\begin{center}
{\Large \bf The contact process with fitness on\\[3mm] random trees
}\\[5mm]

\vspace{0.7cm}
\textsc{Natalia Cardona-Tob\'on\footnote{Departamento de Estadística, Universidad Nacional de Colombia, Carrera 45 No 26-85, 111321 Bogotá,
Colombia. {\tt 	ncardonat@unal.edu.co}}} and \textsc{Marcel Ortgiese*\footnote{Department of Mathematical Sciences, University of Bath, Claverton Down, Bath, BA2 7AY,
United Kingdom, {\tt ma2mo@bath.ac.uk}}}

{\textit{*Corresponding author:	ma2mo@bath.ac.uk}}
\vspace{0.5cm}
 \end{center}

\vspace{0.3cm}

\begin{abstract}
  \noindent 
The contact process is a simple model for the spread of an infection in a structured population. We consider a variant of this process on Bienaymé-Galton-Watson trees, where vertices are equipped with a random  fitness representing  inhomogeneous transmission rates among individuals. In this paper, we establish conditions under which this inhomogeneous contact process exhibits a phase transition. We first prove that if certain mixed moments of the joint offspring and fitness distribution are finite, then the survival threshold is strictly positive. Further, we show that, if slightly different  mixed moments are infinite, then this implies that there is no phase transition and the process survives with positive probability for any choice of the infection parameter. A similar dichotomy is known for the contact process on a Bienaymé-Galton-Watson tree. However, we show that  the introduction of  fitness means that we have to take into account the combined effect of fitness and offspring distribution  to decide which scenario occurs.

\end{abstract}

\section{Introduction}

The contact process on a graph is a popular interacting particle system that describes the spread of  an infection in a structured population. The model is described informally as follows: The structure of the population is encoded by a graph, where vertices represent  individuals that are susceptible to the infection and the edges depict the connections between them. 
Each infected vertex infects each of its neighbours with an infection rate $\lambda > 0$. Moreover, independently each vertex recovers at rate~$1$. The contact process is sometimes also referred to as the susceptible–infected–susceptible (SIS) epidemic model.
	
The process is monotone in the infection rate $\lambda$ and it is therefore natural \nat{to ask} whether there exists a phase transition. 
For an infinite rooted graph, there are two critical values  of interest $0\leq \lambda_1 \leq \lambda_2$, which determine different regimes: in the \textit{extinction} phase, for $\lambda \in (0,\lambda_1)$, the infection becomes extinct in finite time almost surely. In the \textit{weak survival} phase, when  $\lambda\in (\lambda_1, \lambda_2)$, the infection survives forever with positive probability, but the {root} is infected only finitely many times almost surely. Finally, in the \textit{strong survival }phase,  for $\lambda \in (\lambda_2, \infty)$, the infection also survives forever with positive probability, however in this regime the root is infected infinitely many times with positive probability. 
On the other hand, for a finite graph, the infection dies out almost surely in finite time, so the phase transition instead describes for how long the process survives. A typical example would be the transition between a polynomial and an \nat{exponential} long time (expressed in terms of the size of the graph).

The analysis of the contact process is a classic topic, going back to the early results by Harris \cite{harris1974contact}, who showed that, assuming that the process starts with a single infection, there exists a critical value $\lambda_c(\mathbb{Z}^d)\in (0,\infty)$ such that if $\lambda <\lambda_c(\mathbb{Z}^d)$ the process dies out almost surely, whereas if $\lambda >\lambda_c(\mathbb{Z}^d)$ the infection survives with positive probability. Later, Bezuidenhout and Grimmett \cite{bezuidenhout1990critical} showed that the infection dies out at the critical value $\lambda=\lambda_c(\mathbb{Z}^d)$. According to \cite[page 26]{valesin2024}, it is possible to show that $\lambda_c(\mathbb{Z}^d)=\lambda_2$. In other words, if $\lambda\le \lambda_c(\mathbb{Z}^d)$ the process dies out while if $\lambda>\lambda_c(\mathbb{Z}^d)$ the process survives strongly.

See also 
the book of  Liggett \cite{liggett2013stochastic} for a general introduction to the topic.
The same is not true for graphs where neighbourhoods from a fixed vertex grow faster with distance. 
Pemantle \cite{pemantle1992contact} showed for the infinite $d$-ary tree $\mathbb{T}_d$
with $d \geq 3$ that the contact process
satisfies $0<\lambda_1(\mathbb{T}_d)<\lambda_2(\mathbb{T}_d)<\infty$. This result was later extended to the case $d=2$ by Liggett \cite{liggett1996binary}, see also~Stacey \cite{stacey1996existence} for a short proof which works for all $d\geq2$. The contact process has also been studied on certain non-homogeneous classes of graphs. Chatterjee and Durrett \cite{chatterjee2009contact} considered the contact process on  power
law random graphs 
and showed that the infection exhibits long survival for any $\lambda>0$
contradicting mean-field calculations as previously obtained in \cite{pastor2001, pastor2001b}. 

More recently, the question about the existence of a phase transition for the contact process 
on Bienaymé-Galton-Watson (BGW) trees was settled:
Huang and Durrett \cite{huang2018contact} showed that for the contact process on BGW trees 
the critical value for local survival  is $\lambda_2=0$ if the offspring distribution  {$\mathscr{L}(\xi)$} is subexponential, i.e.\ if  $\mathbb{E}[e^{c\xi}]=\infty$ for all $c>0$. Shortly afterwards, Bhamidi et al. \cite{bhamidi2021survival} proved that on BGW trees, $\lambda_1>0$ if the offspring
distribution  {$\mathscr{L}(\xi)$} has an exponential tail, i.e. if $\mathbb{E}[e^{c\xi}]<\infty$ for some $c>0$.  These two results give a complete characterisation on the existence of extinction phase on BGW
trees, while the question of whether $\lambda_1 < \lambda_2$ in general remains open.

A natural generalization of the contact process is to introduce 
inhomogeneities into the graph by associating a random fitness to each vertex that
influences how likely the vertex is to receive and to pass on the infection.
Peterson  \cite{peterson2011contact} introduced the contact process with inhomogeneous weights on the  complete graph so that the infection rate between vertices $i$ and $j$
 is $\lambda \mathcal{F}_i \mathcal{F}_j  / n$, where $\lambda >0$ is a parameter, $n$ is the graph size and $\mathcal{F}_i$ and $\mathcal{F}_j$ are the (random) fitness values 
of vertices $i$ and $j$. 
More precisely, under a second moment assumption on the weights, he proved that there is a phase transition at $\lambda_c>0$ such that for $\lambda<\lambda_c$ the contact process dies out in logarithmic time, and for $\lambda>\lambda_c$ the contact process lives for an exponential amount of time. 
The way that the infection rates were chosen by Peterson  was inspired by inhomogeneous random graphs as introduced by Chung and Lu \cite{Chung2002TheAD}. 
Xue  \cite{xue2015lattice, xue2019lattice} studied the contact process with random vertex weights with bounded support on oriented lattices. In particular, the author  investigates in \cite{xue2015lattice} the asymptotic behaviour of the critical value when the lattice dimension grows. Later, Pan et al. \cite{xue2017regular} extended his result to the case of regular trees.  The reader is also referred to \cite{xue2013regular, xue2016recovery} for further results about the contact process with random weights and bounded support on regular graphs. In all these cases the fitness is bounded 
from above, while  we would also like to consider the case when this does not hold.

In this article, we are interested in understanding the interplay between the inhomogeneous contact process as considered by~\cite{peterson2011contact} with an inhomogeneous graph such as the BGW tree.
We focus on BGW trees, since these can be often used to describe 
the local geometry of random graphs
and standard techniques should apply to translate our results to random graphs.
A natural interest is then to study the phase diagram of this model
and to understand how the extra randomness changes the characterisation of whether
a phase transition occurs or not. A particular challenge is to differentiate the 
effect of having unusually large neighbourhoods versus having a unusually large fitness.

Our first result shows the existence of a phase transition when the distribution of the offspring and the fitness satisfy a certain mixed moment condition. In other words, for $\mathcal{F}$ denoting the fitness associated to a vertex, we prove that if $\mathbb{E}[\mathcal{F}(1+c\mathcal{F})^{\xi}]<\infty$ for some $c>0$, then the survival threshold is strictly positive. The second result tells us that if for some $\vartheta>0$ we have
$	\mathbb{E}\big[(1+\mathcal{F})^{c\xi}\mathbf{1}_{\{\nat{\xi \ge 7}, \, \mathcal{F}\ge \vartheta\}}\big]= \infty$ for all $c>0$,  \nat{then} the process always survives strongly.

In our setting, the fitness has a significant effect on the behaviour of the model as a whole. 
 For instance, if we consider the standard contact process on a BGW tree, where the offspring distribution has exponential tails, then, as mentioned earlier, the process exhibits a phase transition, so in particular the process dies out for small 
 $\lambda$. 
 However, if the  distribution of the fitness values is sufficiently heavy-tailed, then
irrespectively of how light the tails of the offspring distribution are, 
  the process no longer has a phase of extinction. 
 In other words, the presence of fitness guarantees that the infection survives forever with positive probability regardless of the value of $\lambda$.
Conversely, we can also consider the case when the fitness value is a decreasing function of the degrees, so 
that the influence of  high degrees is weakened. A special case of this scenario is the penalised contact process, recently introduced
and analysed in~\cite{zsolt2023}. In this case, the presence of the fitness can induce a phase transition, even if the standard contact 
process would not show one.

\section{Definitions and main results}

In this section, we 
formally introduce our model and state and discuss our main results. 
We consider a random weighted tree $(\mathcal{T}, \mathbb{F}(\mathcal{T}))$. If  $V(\mathcal{T})$
denotes the set of vertices, then we associate to 
each vertex $v \in \mathcal{T}$ a fitness value $\mathcal{F}_v$, so that $\mathbb{F}(\mathcal{T})
= (\mathcal{F}_v)_{v \in V(\mathcal{T})}$.
We assume that the (marginal) distribution of $\mathcal{T}$ is that of a standard Bienaym\'e-Galton-Watson (BGW) tree with a root vertex $\rho$. 
Let $\xi_v$ denote the number of children of $v$, which is the number of neighbours that is further away from the root than $v$ itself.
Let $(\xi,\mathcal{F})$ be a random variable taking values in $\N_0 \times (0,\infty)$, \nat{where we set $\mathbb{N}_0:= \{0,1,\dots\}$}. 
We assume that  $(\xi_v, \mathcal{F}_v)_{v \in V(\mathcal{T})}$ are independent and have the same distribution as $(\xi,\mathcal{F})$, where we allow for correlations between the number of children and the fitness.
Throughout, we will assume that $\mu := \mathbb{E}[\xi] > 1$, so that $\mathcal{T}$ is infinite with positive probability and also that $\mathcal{F} > 0$ almost surely.
We will refer to $(\mathcal{T}, \mathbb{F}(\mathcal{T}))$ constructed in this way as a weighted BGW-tree.

\begin{definition} Given a weighted BGW-tree $(\mathcal{T}, \mathbb{F}(\mathcal{T}))$, the \emph{inhomogeneous contact process} on  $(\mathcal{T}, \mathbb{F}(\mathcal{T}))$ is a continuous-time Markov chain  with state space $\{0,1\}^{V(\mathcal{T})}$, where a vertex is either infected (state 1) or healthy (state 0). We denote the process by 
$$ (X_t)_{t\geq 0} \sim \textnormal{\textbf{CP}}\big(\mathcal{T};{\mathbf{1}_{A}}\big),$$
if ${\mathbf{1}_{A}}$ is the initial configuration, where the vertices in  $A\subset V(\mathcal{T})$ are initially infected. Given $\lambda >0$, the process evolves according to the following rules:
\begin{itemize}
	\item[(i)] For each $v \in V(\mathcal{T})$ such that $X_t(v)=1$, the process $X_t$ becomes $X_t -\mathbf{1}_{v}$ at rate~1.
	\item[(ii)] For each $v\in V(\mathcal{T})$ such that $X_t(v)=0$, the process $X_t$ becomes $X_t + \mathbf{1}_{v}$ at rate 
	$$ \lambda\sum_{v\sim v'} \mathcal{F}_{v'} \mathcal{F}_v X_t(v'),$$
	 where $\mathcal{F}_v$ and $\mathcal{F}_{v'}$ are the fitness values associated to $v$ and $v'$ and $v\sim v'$ means that vertices $v$  and $v'$ are connected by an edge in $\mathcal{T}$. 
\end{itemize}
\end{definition}

By taking $\cF_v = 1$ in the above definition, we recover the definition of the classic contact process on a BGW-tree with offspring distribution $\mathcal{L}(\xi)$, \nat{where $\mathcal{L}(\xi)$ denotes the law of $\xi$.}

\textbf{Notation: } We use the notation \textbf{0} for  the all-healthy state, i.e.\ $\textbf{0}=\mathbf{1}_\emptyset$.  We also often identify any state $\{ 0,1\}^{V(\mathcal{T})}$ with the subset of $V(\mathcal{T})$ consisting of the vertices that have state $1$ (i.e.\ the infected vertices). For example, when we write $\rho \in X_t$, it means that the root of the tree is infected at time $t$.
We use the conditional probability measure $\mathbb{P}_{\mathcal{T}, \mathbb{F}}(\cdot):=\mathbb{P}\big(\cdot | \ \mathcal{T}, \mathbb{F}(\mathcal{T})\big)$ with associated expectation operator  $\mathbb{E}_{\mathcal{T}, \mathbb{F}}[\cdot]:=\mathbb{E}\big[\cdot | \ \mathcal{T}, \mathbb{F}(\mathcal{T})\big]$. For any set $A$, we write $|A|$ to denote its cardinality.

Now, we define the critical values for the infection parameter $\lambda$. Given the tree $(\mathcal{T}, \mathbb{F}(\mathcal{T}))$ \nat{and the process starting with only the root infected}, we define the threshold between  extinction and weak survival by
\begin{equation*}
	\lambda_1(\mathcal{T},\mathbb{F}(\mathcal{T})):= \inf\left\{\lambda: \ \mathbb{P}_{\mathcal{T},\mathbb{F}}\big( X_t\not= \textbf{0} \ \text{for all}\ t\geq 0  \big) >0\right\},
\end{equation*}
and the weak-strong survival threshold by
\begin{equation*}
	\lambda_2 (\mathcal{T},\mathbb{F}(\mathcal{T})):=\inf \big\{\lambda:\   \mathbb{P}^{\{\rho\}}_{\mathcal{T},\mathbb{F}}\big( \mbox{for any } s \geq 0 \mbox{ there exists } t \geq s \, : \, \rho \in X_{t} \big)>0\big\}.
\end{equation*}
By using the same arguments as in Proposition 3.1 in  Pemantle~\cite{pemantle1992contact}, we see that $\lambda_1(\mathcal{T},\mathbb{F}(\mathcal{T}))$ and $\lambda_2(\mathcal{T},\mathbb{F}(\mathcal{T}))$  are constant for almost every $(\mathcal{T},\mathbb{F}(\mathcal{T}))$ conditioned on $|\mathcal{T}|=\infty$.  We let $\lambda_1$ and $\lambda_2$ denote these two constants.

Throughout this paper, we shall suppose that $\mu =\mathbb{E}[\xi]<\infty$. 
We will show in Theorem~\ref{teo_nonexplosion}
that this condition guarantees that the set of infected vertices at every time is finite almost surely.

The main question we want to answer is whether the 
introduction of the fitness changes whether the model exhibits a phase transition or not. 
In Section~\ref{sec_properties} we will see that 
the process is monotone in $\lambda$ as well as in the fitness values $\cF_v$. In particular, this allows us to compare the inhomogeneous process 
to the classical contact process on a BGW tree. 
More precisely, if there \nat{exists} $\vartheta > 0$ such that 
$\mathbb{P}(\mathcal{F} \geq \vartheta) = 1$, then if $\xi$ is such that the standard contact process 
on a BGW tree with offspring distribution $\mathcal{L}(\xi)$ has no phase transition, then 
the same is true for the inhomogeneous process. 
Conversely, if there \nat{exists} $\kappa > 0$ such that 
$\mathbb{P}(\mathcal{F} \leq \kappa) = 1$, then if $\xi$ is such that the standard contact process 
on a BGW tree with offspring distribution $\mathcal{L}(\xi)$ dies out for small $\lambda$, then the same is true for the inhomogenous model.

The following theorem gives a sufficient criterion 
that guarantees the existence of a subcritical phase.

\begin{theorem}\label{teo_exptail}
	Consider the inhomogeneous contact process  on the tree $(\mathcal{T}, \mathbb{F}(\mathcal{T}))$. Suppose that only the root of the tree is initially infected. Assume that 
	\begin{equation}\label{eq:exp_tails}\tag{\bf{A}}
		\mathbb{E}\left[\mathcal{F}(1+c\mathcal{F})^{\xi}\right]<\infty \quad \text{for some}\quad  c>0.
	\end{equation} Then there exists a (deterministic) $\lambda_0> 0$ such that for all $\lambda < \lambda_0$, the process dies out  almost surely. 
\end{theorem}

Our second main theorem gives a sufficient criterion for the process to always survive strongly, so that there is no
 phase transition.

\begin{theorem}\label{teo_noexptail}
	Consider the inhomogeneous contact process on the tree $(\mathcal{T}, \mathbb{F}(\mathcal{T}))$. Suppose that only the root of the tree is initially infected. Assume that 
 $\mathbb{P}(\xi\ge {\color{black}\kzero})>0$ and for some $\vartheta>0$
 \begin{equation}\tag{\bf{B}}\label{eq:offspringfitness}
   \mathbb{E}\big[(1+\mathcal{F})^{c\xi}\mathbf{1}_{\{{\color{black}\xi\ge \kzero}, \, \mathcal{F}\ge \vartheta\}}\big] = \infty, \qquad \text{for all}\quad  c>0.			
  		\end{equation}
 Then $\lambda_1=\lambda_2=0$, i.e.\ the process survives strongly for any $\lambda>0$.
\end{theorem}

\begin{remark}[Existence of phase transition]
The condition~\eqref{eq:exp_tails} on its own does not imply a phase transition. It is still feasible that if very small fitness values are allowed, then the process dies out for all $\lambda > 0$. 
However, if $\mathcal{F}$ is bounded away from $0$, i.e.\ there exists $\vartheta > 0$ such that $\mathbb{P}(\mathcal{F} \geq \vartheta) = 1$, then 
by comparison with the standard contact process, the inhomogeneous model 
\nat{survives (strongly) for all $\lambda$ sufficiently large,} so that we can deduce the existence of a phase transition in this case. 
\end{remark}

\begin{remark}[Comparison to classical case]
In the special case when $\mathcal{F}$ is constant, 
we recover the results already known in the literature.
More precisely, if w.l.o.g.\ $\mathcal{F}\equiv1$, then 
 Condition \eqref{eq:offspringfitness} reduces to the condition $\mathbb{E}[e^{\tilde{c}\xi}]=\infty$ for all $\tilde{c}>0$, which is known by Huang and Durrett \cite[Theorem 1.4]{huang2018contact} to be a sufficient condition for the lack of a phase transition. 
 Similarly, the condition~\eqref{eq:exp_tails}
 corresponds to 
 $\mathbb{E}[e^{\tilde{c}\xi}]<\infty$ with $\tilde{c} = \log (1+c)$, which is known by 
 Bhamidi et al. \cite[Theorem 1]{bhamidi2021survival} to imply survival for small $\lambda$.
 \end{remark}

\begin{remark}[Effect of the inhomogeneous fitness]
Note that the inhomogeneous fitness can 
either have the effect to prevent a phase transition if particularly large values are allowed or if small values are allowed it can introduce a phase transition 
when the classical model does not exhibit one. 
The following two examples are covered by our results:
\begin{enumerate}
    \item[(i)]
   Assume that  fitness and  offspring distribution are independent and assume that $\xi$ has exponential moments (but $\mathbb{P}(\xi \geq {\color{black}7}) > 0$) and $\mathcal{F}$ has unbounded support and 
satisfies 
\[ \mathbb{E}\big[(1+\mathcal{F})^{c}\mathbf{1}_{\{\mathcal{F}\ge \vartheta\}}\big] = \infty, \qquad \nat{\text{for some } \vartheta>0 \text{ and }} \text{for all }  c>0.\]	Then the standard contact process has a phase transition, while Theorem~\ref{teo_noexptail} guarantees that the inhomogeneous 
contact process is always supercritical. 
This also includes the case of bounded degree graphs, such as a $k$-regular tree with $k \geq {\color{black}7}$. 
Obviously, condition~\eqref{eq:offspringfitness} covers many other examples (including cases when there dependencies between fitness values and offspring numbers).
    \item[(ii)]
    Consider the inhomogeneous contact process with fitness values chosen as $\mathcal{F}=\min\{1, \xi^{-\alpha}\}$ for $\alpha\in (0,\infty)$. This is a particular case of the \textit{degree-penalised contact process} with  product kernel  $\min\{1 , (\xi_u \xi_v)^{-\alpha}\}$ introduced by Bartha et al. \cite{zsolt2023}. Note that in this case, using that $(1+1/x)^x \leq e$  for $x > 0$, we have 
    \begin{equation}\label{eq:penalised}\begin{split}
        \mathbb{E}\left[\mathcal{F}(1+c\mathcal{F})^{\xi}\right] \le \mathbb{E} \left[\xi^{-\alpha} (1+c\,\xi^{-\alpha})^{\xi}\right] \le \mathbb{E}\left[e^{c\xi^{1-\alpha}}\right].
    \end{split}\end{equation}    
    So if the right hand side is finite, then 
     according to Theorem \ref{teo_exptail} the inhomogeneous contact process exhibits a subcritical phase. 
     A first example is when $\alpha\geq 1$, regardless of the distribution of $\xi$. 
     Another example, would 
be the case when $\xi$ has a stretched exponential distribution, i.e.\ $\mathbb{P}(\xi=k) \sim  \exp(-k^{\beta})$ as $k \rightarrow \infty$ for some $\beta \in (0,1)$ and $k\ge 1$. Then, if $\beta \geq 1-\alpha$, the right hand side of~\eqref{eq:penalised} is finite and there is a phase transition, even though the classical contact process has no phase transition.
   
   Unfortunately, note that Theorem \ref{teo_noexptail} does not apply in this setting as the expectation in condition~\eqref{eq:offspringfitness} is always restricted to $\mathcal{F}\geq \vartheta$, which under the assumption on $\mathcal{F}$ means that $\xi$ is restricted to values  bounded from above. As also $\mathcal{F} \leq 1$, this implies that the expectation in~\eqref{eq:offspringfitness} is always finite.
   Bartha et al.~\cite{zsolt2023} have showed 
   results (using different techniques) that provide a much more complete picture in this case. In particular, they proved that the degree-penalised contact process with product kernel $\min\{1 , (\xi_u \xi_v)^{-\alpha}\}$ always exhibits a subcritical phase if $\alpha\ge 1/2$. On the other hand, if $\alpha<1/2$ and the offspring distribution has tails heavier than the stretched-exponential with $\beta=1-2\alpha$, then the contact process  survives strongly for any $\lambda>0$ (see \cite[Theorem 2.1 and 2.2]{zsolt2023}).
\end{enumerate}
\end{remark}

\begin{remark}
Note that there are cases not captured by either of the two conditions 
in Theorem~\ref{teo_exptail} and Theorem~\ref{teo_noexptail}.
For example, 
suppose $\xi$ and $\mathcal{F}$ are chosen independently with distributions  $$
\begin{aligned}\mathbb{P}(\xi =k) & = \eta_1\, e^{-\frac{1}{4}k^2}, \quad  \quad k \in  \mathbb{N},\\ \mathbb{P}(\mathcal{F}\geq f)& =  e^{-\frac{1}{4}(\log (1+f))^2}, \quad \quad  f \geq 0,\end{aligned}$$
where $\eta_1$ is a normalizing constant.
It is not difficult to see that both distributions fulfill neither Assumption {\eqref{eq:exp_tails}} nor {\eqref{eq:offspringfitness}} of  Theorems  \ref{teo_exptail} and \ref{teo_noexptail}. In fact, for  $c>0$, note that
\[\begin{split}f(1+cf)^k \mathbb{P}(\xi= k) \mathbb{P}(\mathcal{F} \geq f) &= \eta_1  f(1+cf)^k e^{-\frac{1}{4}k^2} e^{-\frac{1}{4}(\log (1+f)^2}\\ & \geq  \eta_1   c^k f^{k+1}  e^{-\frac{1}{4}k^2} e^{-\frac{1}{4}(\log (1+f))^2}. \end{split}\]
Taking $k=\lfloor \log (1+f)\rfloor$ the right-hand side of the previous expression goes to $\infty$ as $f\to \infty$, so $\mathbb{E}[\mathcal{F} (1+c\mathcal{F})^{\xi}]$ cannot be finite as
$$\mathbb{E}\big[\mathcal{F} (1+c\mathcal{F})^{\xi}\big] \geq \mathbb{E}\big[f(1+cf)^{\xi}\big]\mathbb{P}(\mathcal{F}\geq f) \geq f(1+cf)^k \mathbb{P}(\xi= k) \mathbb{P}(\mathcal{F} \geq f).$$
Furthermore, Condition \eqref{eq:offspringfitness} is also not satisfied, as
	\begin{equation*}
	 \limsup_{\substack{f+k\to \infty\\ \, \nat{k\ge \kzero}, \,  f\ge \vartheta}} \frac{\log \big(\mathbb{P}(\xi=k) \mathbb{P}(\mathcal{F}  \geq f)\big)}{k\log (1+f)}= 	 \limsup_{\substack{f+k\to \infty\\ \, \nat{k\ge \kzero}, \,  f\ge \vartheta}} \frac{\log (\eta_1)}{k \log (1+f)}-\frac{1}{4}\frac{k}{\log (1+f)} -\frac{1}{4} \frac{\log (1+f)}{k} = -\frac{1}{2},
\end{equation*}
for any $\vartheta>0$, by Lemma~\ref{lem:equivcond} below.

In general, 
it is not clear what  the correct moment condition is to give a complete characterization of 
when a phase transition exists. We believe that this problem requires different techniques to those developed here and is currently on-going research. 
\end{remark}

{\bf Main ideas of proofs.} 
In this section, we give a short overview of the proofs of the main theorems before giving the full proofs. 

The proof of Theorem \ref{teo_exptail} is 
 an adaptation of a simplified version of the proof from Bhamidi et al. \cite{bhamidi2021survival}. See also the  lecture notes by Daniel Valesin \cite{valesin2024}, where the simplification is presented for the standard contact process. First, we use a recursive analysis on BGW trees that allows us to control the expected survival times. To this end, we consider the contact process on  the finite tree $\mathcal{T}_L$ which corresponds to the restriction of $\mathcal{T}$ to the first $L$ generations. The first goal is to show that, for  small enough $\lambda$, the expected survival time of $ \text{\textbf{CP}}(\mathcal{T}_L; \mathbf{1}_{\rho})$ is bounded from above uniformly in $L$. As in~\cite{bhamidi2021survival}, we use a coupling, where
we add an extra vertex only adjoined to the root that is always infected. 
In this way, the process on the subtrees rooted in the children of the root
can by independence be compared to the full process on a tree (with extra root) 
restricted to $L-1$ vertices (see Lemma \ref{lem_stationdistr} below).
The recursion gives a bound on the expected survival time on $\mathcal{T}_L$ that is uniform in $L$. Thus, using monotone convergence 
we get a bound on the expected survival time 
on the full tree which immediately implies that the inhomogeneous contact process dies out almost surely.

The proof of Theorem \ref{teo_noexptail} is based on  techniques developed by Pemantle in \cite{pemantle1992contact}. He used these techniques to show an upper bound for the threshold value $\lambda_2$ for the contact process with constant fitness defined on BGW trees with constant offspring
number, but also for those with an stretched exponential distribution.
More recently, this strategy  was extended by Durrett and Huang in~\cite[Theorem 1.4]{huang2018contact} in order to show that $\lambda_2=0$ if $\mathscr{L}(\xi)$ has subexponential tails. The same overall strategy holds in our case, but we need to take into account the effect of a possibly large fitness and thus large transmission rates.

First we estimate the survival time for the contact process with fitness on a finite star. For the case with constant fitness, Berger et al. in \cite{berger2005spread} showed that for a star of size $d$, there are universal constants $C>0$ and $c>0$ such that if $d$ is larger than $C/\lambda^2$, then the infection survives on the star for a time at least $e^{c\lambda^2d}$ with high probability.

Now, in the case with fitness  we prove that, if the root of the star is initially infected,  it will keep a large number of infected leaves for a long time with probability very close to 1 if the fitness or the size of the star is large enough (see Lemma \ref{lem_star}). 
The inhomogeneous contact process on a star, where the center has fitness larger than $f$ and the leaves have constant fitness equal to 1, can be lower bounded by a standard contact process where the infection rate is $\lambda f$. 
In particular, we need to take extra care to get the right estimates that also hold if only \nat{$\lambda$}$f$ (i.e.\ the effective transmission rate) is large, but for example the size of star is small.
With the right estimates on the star at hand, we study the contact process in a star where a path is added to some \nat{leaf} of the star (see  Lemma \ref{lem_starchain})
and look at the probability to pass the infection from the start to the end vertex of that path.

The overall strategy  
for the proof of Theorem \ref{teo_noexptail}, where we need  to show that the root is infected at  large times, 
is as follows: We first assume that the root 
is a star, i.e.\ it has either a sufficient large number of neighbours and/or  large fitness. 
Then, the above start lemmas allows us to push the 
infection sufficiently deep into the tree, where 
we can find another star, where the infection 
survives for a long time, so that the infection 
can be brought back to the root at the end.

{\bf Overview of structure.} The remaining paper is structured as follows. In Section \ref{sec_properties}, we introduce the graphical representation for the contact process
and we state further useful properties. In particular, we will also see that started with a single infected vertex, the set of infected vertices stays finite at all times. Section~\ref{sec_prooftheoexptails} is devoted to the proof of Theorem~\ref{teo_exptail}. In Section \ref{sec_finitestars}, we prove preliminary results for the contact process with fitness on finite stars. Finally, in Section \ref{sec_prooftheononexptails} the proof of Theorem~\ref{teo_noexptail} is completed.

\section{Properties of the inhomogeneous contact process}\label{sec_properties}
In this section we provide an equivalent description of our model by a convenient \textit{graphical representation} based on the construction given in \cite[Chapter 1]{liggett2013stochastic} for the case of constant fitness.  We point out some properties which are direct consequences of the  construction. Further, we show that under the assumption $\mathbb{E}[\xi]<\infty$, the contact process does not explode in finite time almost surely, so that the set of infected vertices at any time is finite almost surely.

Recall that $V(\mathcal{T})$ denotes the set of vertices in $\mathcal{T}$. Denote by $E(\mathcal{T})$ the set of  undirected edges. 
Conditionally on  $(\mathcal{T},\mathbb{F}(\mathcal{T}))$, let $\{N_{v}\}_{v\in V(\mathcal{T})}$ be i.i.d.\ Poisson (point) processes with rate 1 and $\{N_{(v,u)}\}_{(v,u)\in E(\mathcal{T})}$ i.i.d.\ Poisson processes with rate $\lambda\mathcal{F}_{v}\mathcal{F}_{u}$, where all the Poisson processes are mutually independent. 
Then, given any initial condition, the contact process can be defined as follows: an infected vertex $v$ infects another vertex $u$ at  a time $t$ if $t$ is in $N_{(v,u)}$. Similarly, an infected vertex recovers at any time that is in $N_{v}$.

One advantage of the graphical construction is that it provides a joint coupling of the processes with different infection rules or different initial states. We state two useful facts about the contact process that we will use later in our proofs and that are a consequence of using the  graphical representation  and the fact that we can e.g.\ easily couple Poisson processes with different rates.

\begin{itemize}
	\item\textit{Monotonicity in the infection rates. } Let $\mathbb{F}_1(\mathcal{T})=(\mathcal{F}_v^1)_{v\in V(\mathcal{T})}$ and $\mathbb{F}_2(\mathcal{T})=(\mathcal{F}_v^2)_{v\in V(\mathcal{T})}$ be sequences of fitness values such that $\mathcal{F}_v^1\leq \mathcal{F}_v^2$ a.s. for all $v\in V(\mathcal{T})$. Let  $(X_t^1)\sim \text{\textbf{CP}}((\mathcal{T}, \mathbb{F}_1(\mathcal{T})); \mathbf{1}_{A})$ and $(X_t^2)\sim \text{\textbf{CP}}((\mathcal{T}, \mathbb{F}_2(\mathcal{T})); \mathbf{1}_{A})$ starting from the same initially infected set $A\subset V(\mathcal{T})$. Then, we can couple both processes such that for any $t\geq 0$, we have $X_t^1\leq X_t^2$, i.e.\ 
	$$ X_t^1(v)\leq X_t^2(v), \quad \text{for all}\quad v\in V(\mathcal{T}),$$ 
	(see e.g. \cite[Section 1.2]{peterson2011contact} for the case of the contact process with fitness in a completely deterministic graph).
	\item Consider $(X_t^1)\sim \text{\textbf{CP}}(\mathcal{T}; {\mathbf{1}_{A}})$ started from any initially infected set $A\subset V(\mathcal{T})$. Let $\textbf{I}$ be any subset of $[0,\infty)$. Define $(X_t^2)$ to be a process  that has the same initial state, infection and recoveries as $(X_t^1)$, except that the recoveries at a fixed vertex $v\in V(\mathcal{T})$ are ignored at times $t\in \textbf{I}$. Then, we can couple both processes such that for all $t\geq 0$ we have $X_t^1\leq X_t^2$, i.e.
	$$ X_t^1(v)\leq X_t^2(v), \quad \text{for all}\quad v\in V(\mathcal{T}),$$ 
	(see for instance \cite[Lemma 2.2]{bhamidi2021survival} for the case  of the contact process with constant fitness).
\end{itemize}

As a next result we show that if we start with a finite configuration, then almost surely 
the configuration remains finite for all times. Our argument adapts the proof of  Durrett \cite[Theorem 2.1]{durret1995notes} to our setting with the additional difficulty that the underlying graph is random and we have unbounded rates.

Let $r$ be an arbitrary non-negative integer. Let $V_r$ denote the set of vertices in generation~$r$ in the tree $\mathcal{T}$, i.e.\
\begin{equation}\label{eq_Vr}
	V_r=\{v\in V(\mathcal{T}): d(\rho, v)=r\},
\end{equation}
where $d(\cdot,\cdot)$  is the graph distance between two vertices in the tree. Let $t_0$  be a positive number. Define the graph $G_{t_0}$, a spanning subgraph of the BGW tree $\mathcal{T}$,  by saying  that two vertices $u,v$ in $V(\mathcal{T})$ with $d(\rho, u) < d(\rho,v)$ that are neighbours in $\mathcal{T}$
	are also neighbours in $G_{t_0}$ 
	if 
	\[ 
	N_{(u,v)}\big([0,t_0]\big) \geq 1. \]
	Let  $\mathscr{C}_{t_0}(v)$ denote the vertex set of the connected component of $v$ in the graph $G_{t_0}$.
We will first show that if $t_0$ is small enough, then the component sizes in $G_{t_0}$ are finite almost surely.
Then, we will make the connection to the contact process as follows: If we are starting with only the root infected, then by the graphical construction and the fact that we are working on a tree, every vertex that has been infected by time  $t_0$ in the contact process is contained in $\mathscr{C}_{t_0}(\rho)$.

\begin{lemma}\label{lem_finitet0}
Assume that $\mu = \mathbb{E}[\xi]<\infty$. 
If $t_0$ is small enough, then for any finite $\ell \in \N$, 
	\begin{equation*}\label{eq:Compt0}
\mathbb{P}\bigg(\bigg|\bigcup_{v\in \bigcup_{r=0}^{\ell}V_r}\mathscr{C}_{t_0}(v)\bigg|<\infty\bigg)=1.
\end{equation*}
\end{lemma}

\begin{proof}	We start by showing that $|\mathscr{C}_{t_0}(\rho)| < \infty$  almost surely.
For this result, by Borel-Cantelli, it suffices to show that
  \begin{equation}\label{eq:summable} \sum_{r=0}^{\infty}\mathbb{P}\big(\mathscr{C}_{t_0}(\rho)\cap V_r \not= \emptyset\big)<\infty.\end{equation}
We begin by denoting  the set of non-backtracking paths of length $r$ in the tree $(\mathcal{T}, \mathbb{F}(\mathcal{T}))$ started at the root as follows
$$ C_r = \{v=(v_0, \dots, v_r): v_i \in V_i, \ i \in \{0,\dots, r\} \nat{\ \text{and}\ v_i\sim v_{i+1},\ i\in \{0,\dots,r-1\}}\}.$$
Note that with this notation $v_0=\rho$ and the number of paths of length  $r$ corresponds to the number of vertices in generation $r$, that is, $|C_r|=|V_r|$.

Let  $c=(v_0, \dots, v_r) \in C_r$ be a path of length $r$ in $(\mathcal{T}, \mathbb{F}(\mathcal{T}))$. For each $i\in \{0,\dots, r\}$, by definition of $G_{t_0}$, we have that
the probability that the vertices $v_{i-1}$ and $v_i$ are connected  in $G_{t_0}$
is
 $$ 1- \exp\left(-\lambda \mathcal{F}_{v_{i-1}}\mathcal{F}_{v_{i}} t_0\right).$$
By monotonicty, we may assume that $r = 2k$ for some $k \in \mathbb{N}_0$.
Thus, conditioning on $(\mathcal{T}, \mathbb{F}(\mathcal{T}))$,  the probability that there is a vertex in generation $r$ that is also in $\mathscr{C}_{t_0}(\rho)$ 
is
\[\begin{split}
\mathbb{P}_{\mathcal{T},\mathbb{F}}\big(\mathscr{C}_{t_0}(\rho)\cap V_r \not=\emptyset\big) &\leq \sum_{v \in C_r}  \prod_{i=1}^{r}\Big(1- \exp\left(-\lambda \mathcal{F}_{v_{i-1}}\mathcal{F}_{v_{i}}t_0\right)\Big) \\
& \leq \sum_{v \in C_r}  \prod_{i=0}^{k-1}\Big(1- \exp\left(-\lambda \mathcal{F}_{v_{2i}}\mathcal{F}_{v_{2i+1}}t_0\right)\Big).
\end{split}\]
Denote by $\mathbb{T}_{\leq r}$ and $\mathbb{F}_{\leq r}$ the $\sigma$-algebras generated by the BGW tree and the fitness values
up to generation $r$ respectively.
Conditioning on $\mathbb{F}_{\le 2k-3}$ and $\mathbb{T}_{\le 2(k-1)}$ and using the independence structure  of the weighted BGW 
\[\begin{aligned}  \E\big[ \mathbb{P}_{\mathcal{T},\mathbb{F}} & \big(\mathscr{C}_{t_0}(\rho)\cap V_r \not=\emptyset\big) \, |\, \mathbb{F}_{\le 2k -3},\mathbb{T}_{\le 2k-2}\big] \\
& \leq  \sum_{v \in C_{2k-2}}  \prod_{i=1}^{k-2}\Big(1- \exp\left(-\lambda \mathcal{F}_{v_{i-1}}\mathcal{F}_{v_{i}}t_0\right)\Big)
\E\Big[ 
\sum_{w \in C_2}  \Big(1- \exp\left(-\lambda \mathcal{F}_{w_0}\mathcal{F}_{w_1}t_0\right)\Big) \Big] , 
\end{aligned}\]
where $w = (w_0,w_1,w_2)$.
Iterating this procedure gives
\[ \mathbb{P}\big(\mathscr{C}_{t_0}(\rho)\cap V_r \not=\emptyset\big) \leq \E\Big[ 
\sum_{w \in C_2}  \Big(1- \exp\left(-\lambda \mathcal{F}_{w_0}\mathcal{F}_{w_1}t_0\right)\Big) \Big]^k . \]
As the summand in the last expectation is bounded by $1$ and $\mathbb{E}[ |C_2|] \leq \E[ \xi]^2 < \infty$, we
can deduce by dominated convergence that we can choose $t_0$ small enough such that
\[ \E\Big[ 
\sum_{w \in C_2}  \Big(1- \exp\left(-\lambda \mathcal{F}_{w_0}\mathcal{F}_{w_1}t_0\right)\Big) \Big] < 1. \]
Thus, 
$\mathbb{P}\big(\mathscr{C}_{t_0}(\rho)\cap V_{2k} \not=\emptyset\big)$ is summable in $k$. 
Together with the \nat{monotonicity} of the process this implies~\eqref{eq:summable}. 
It follows by the Borel-Cantelli lemma that
 almost surely for $r$ sufficiently large
 $\mathscr{C}_{t_0}(\rho) \cap V_r = \emptyset$ with probability 1, 
 so that $|\mathscr{C}_{t_0}(\rho)|$ is finite almost surely. 
 
 The same argument shows that 
almost surely for each $v \in V_\ell$, we have that 
$\mathscr{C}_{t_0}(v)$ restricted to the vertices in the subtree of $\mathcal{T}$
rooted at $v$ is finite. This immediately implies the statement of the lemma, 
as $\bigcup_{r = 0}^\ell V_r$ is finite almost surely.
\end{proof}

\begin{theorem}\label{teo_nonexplosion}
	Assume that $\mu = \mathbb{E}[\xi]<\infty$. Consider the inhomogeneous contact process $(X_t)\sim \normalfont{\textbf{CP}}(\mathcal{T}, \mathbf{1}_A)$ on $(\mathcal{T},\mathbb{F}(\mathcal{T}))$
	started with a finite set $A \subset V(\mathcal{T})$ initially infected. Then
	$$ \mathbb{P}\big(|X_t|<\infty, \ \forall t\geq 0\big) = 1.$$
\end{theorem}

\begin{proof}
	Let $t_0$ be as given in Lemma~\ref{lem_finitet0}. 
	Since the initial configuration of the contact process is finite, we can find a sufficiently large and finite $k$ such that $A \subset \bigcup_{r = 0}^k V_r$, so that 
	by the graphical construction we have that
 $X_t\subset \cup_{v\in\cup_{r=0}^k V_r} \mathscr{C}_{t_0}(v)$  for all $t\in [0,t_0] $. Then, according to Lemma \ref{lem_finitet0}, we have that
	$$ \mathbb{P}\big(|X_t|<\infty, \ \forall t\in [0,t_0]\big)=1.$$
	Let us now define a spanning subgraph $G_{2t_0}$ of $\mathcal{T}$ as follows:
	two neighbouring vertices $u, v$ in $\mathcal{T}$ with $d(\rho,u) < d(\rho,v)$ are connected if
	$$ 
	  N_{(u,v)}\big([t_0,2t_0]\big)\geq 1.$$
	  As before, let $\mathscr{C}_{t_0}(\rho)$ be  the connected component of  $\rho$ in the graph $G_{t_0}$. For each $v\in \mathscr{C}_{t_0}(\rho)$, we denote by $\mathscr{C}_{2t_0}(v)$ the connected component of $v$ in the graph $G_{2t_0}$. We observe that  by the graphical construction $$X_t\subset \bigcup_{v\in \mathscr{C}_{t_0}(\rho)}\mathscr{C}_{2t_0}(v), \quad \quad  \text{for all}\quad  t\in [t_0,2t_0].$$
 Then we can argue that by the independence of $G_{t_0}$ and $G_{2t_0}$,
and appealing again to Lemma \ref{lem_finitet0}, that the right-hand side is finite a.s., which implies
$$ \mathbb{P}\big(|X_t|<\infty, \ \forall t\in [t_0,2t_0]\big)=1.$$
Iterating the argument we can conclude the proof.
\end{proof}

\section{Proof of Theorem \ref{teo_exptail}}\label{sec_prooftheoexptails}

This section is devoted to proving Theorem \ref{teo_exptail}. 
The general strategy follows the arguments used in  Bhamidi et al. \cite[Theorem 1]{bhamidi2021survival}, 
although  the presence of fitness leads to significant changes.
Moreover, we do not adapt their arguments directly, but follow a simplification of the arguments that was suggested to us by Daniel Valesin (see also his \nat{book} \cite{valesin2024}).

First we set up some extra  notation for this section. Let $D$ denote the degree of the root $\rho$ and let $v_1, \dots,v_D$ be the children of $\rho$ with fitness $\mathcal{F}_{v_1},\dots, \mathcal{F}_{v_D}$.
We denote by $\mathcal{T}_L$ the 
tree obtained by  restricting the tree $\mathcal{T}$ to the first $L$ generations. Unless specified otherwise, we suppress the weights in our notation, but implicitly inherit the weights of the vertices from the original weighted tree  $(\mathcal{T}, \mathbb{F}(\mathcal{T}))$.
For $i = 1, \ldots, D$, let $\mathcal{T}_{L-1}^{v_i}$ 
be the subtree of $\mathcal{T}_L$ consisting of all descendants of $v_i$ (including $v_i$)  rooted in $v_i$.
Denote by $\mathcal{T}_L^+$ the 
tree $\mathcal{T}_L$, but with an extra parent  $\rho^+$ attached to the vertex $\rho$. 
We will assume 
that $\rho^+$ 
has constant fitness, i.e.\ $\mathcal{F}_{\rho^+}=1$.  
Moreover, we denote by $\mathcal{T}_{L-1}^{+, v_i}$ the subtree of $\mathcal{T}_L$ consisting of $\rho$ and all the descendants of $v_i$ (including $v_i$ itself).

Given any weighted tree $T$,  a subset $A$ of vertices of $T$ and a vertex $v$ in $T$, we denote by 
$\mathbf{CP}_v(T;\mathbf{1}_A)$ the inhomogeneous contact process that follows the same dynamics as $\mathbf{CP}(T;\mathbf{1}_A)$, but where vertex $v$ is treated as permanently infected.
Since the state of the vertex $v$ never changes, we will throughout specify the state of $\text{\textbf{CP}}_{v}(T; \mathbf{1}_A)$ by specifying the state restricted to $T \setminus\{v\}$, so that 
${\normalfont{\textbf{0}}}$ denotes the state where the only infected vertex is $v$.
If the initial condition is irrelevant, we drop the notation $\mathbf{1}_A$.

\begin{lemma}\label{lem_stationdistr}
Suppose that $\xi$ and $\cF_\rho$ satisfy Assumption~\eqref{eq:exp_tails}, i.e.\ there exists $c >0$
such that $M = \E [ \mathcal{F} (1 + c \mathcal{F})^{\xi}] < \infty$.
Let $R_L$ be the first time when ${\normalfont{\text{\textbf{CP}}(\mathcal{T}_L; \mathbf{1}_{\rho})}}$
	 reaches state {\normalfont{\textbf{0}}}. Then there exists a constant $\lambda_0>0$ 
	  such that for any $\lambda\leq \lambda_0$ and $L$, 
$$\mathbb{E} \big[\mathcal{F}_{\rho}R_L\big] \leq M.$$
\end{lemma}

\begin{proof}
	 We denote by $(\widetilde{X}_t) \sim \widetilde{\textbf{CP}}_{\rho^+}(\mathcal{T}_L^+; \mathbf{1}_{\rho})$  the following modification  of the contact process $\textbf{CP}_{\rho^+}(\mathcal{T}_L^+; \mathbf{1}_{\rho})$. The process $(\widetilde{X}_t)$ shares the same infection and recovery clocks as $(X_t)$ except in the root $\rho$. A recovery attempt at $\rho$ at time $t$ is only valid if $\widetilde{X}_t = \mathbf{1}_{\rho}$, 
	so when there are no other infected vertices apart from $\rho$ and $\rho^+$.
	 Let $S_L$ and $\widetilde{S}_L$ be the first excursion time when the  process  $\textbf{CP}_{\rho^+}(\mathcal{T}_L^+; \mathbf{1}_{\rho})$ and the modified process  $\widetilde{\textbf{CP}}_{\rho^+}(\mathcal{T}_L^+; \mathbf{1}_{\rho})$ reach state \textbf{0}, respectively.  
 An excursion to \textbf{0} of  $(\widetilde{X}_t) \sim \widetilde{\textbf{CP}}_{\rho^+}(\mathcal{T}_L^+; \mathbf{1}_{\rho})$ started from the initial configuration $\widetilde{X}_0 = \mathbf{1}_{\rho}$ can be described as follows:
	\begin{itemize}
		\item Possibility 1. We terminate if $\rho$ recovers before infecting any of its children. The probability that $\rho$ recovers before infecting any of its children conditionally on  $(\mathcal{T}_L, \mathbb{F}(\mathcal{T}_L))$ is given by
		\begin{equation*}
			p:= \frac{1}{1 +  \lambda \mathcal{F}_{\rho}\sum_{j=1}^{D}\mathcal{F}_{v_j}}.
		\end{equation*}
		Furthermore,  the expected waiting time to see the recovery of $\rho$ is  
		\[ \frac{1}{1 +  \lambda \mathcal{F}_\rho \sum_{j=1}^{D}\mathcal{F}_{v_j}}=p,\] conditionally on the fitness and on the event that the recovery happens first.
		\item Possibility 2. The root $\rho$ infects any of its children,  say $v_i$, before the recovery of~$\rho$.  
  The vertex $v_i$ is selected with probability proportional to the  fitness of the root's children, i.e.\ with probability 
		\begin{equation*}
\frac{ \mathcal{F}_{v_i}}{\sum_{j=1}^{D}\mathcal{F}_{v_j}}.
		\end{equation*}
        {\color{black}Let $S_i$ be the first time that $\mathbf{CP}_\rho(\mathcal{T}_L,\mathbf{1}_{v_i})$
        reaches the all-healthy state on $\cup_{i=1}^D \mathcal{T}_{L-1}^{v_i}$.} Then the  expected waiting time until all subtrees of the root have recovered is $\mathbb{E}_{\mathcal{T}_L, \mathbb{F}}[S_i]$, conditionally on vertex $v_i$ being infected first. When this excursion finishes, we come back to the initial state and follow either possibility 1 or 2.
	\end{itemize}
Note that we have reached state $\mathbf{0}$ as soon as the process follows possibility 1.
Hence, by splitting according to the number of times the process follows possibility 2 before finally taking possibility 1,
we obtain
the expected excursion time to \textbf{0} of $(\widetilde{X}_t)$, given the tree $({\mathcal{T}_L, \mathbb{F}(\mathcal{T}_L)})$,
as
	\begin{eqnarray*}
		\mathbb{E}_{\mathcal{T}_L, \mathbb{F}}\Big[\widetilde{S}_L\Big]  &=& \sum_{k=0}^{\infty} p \left(1-p \right)^k \left[(k+1) p +  \frac{k }{ \sum_{i=1}^{D}  \mathcal{F}_{v_i}} \sum_{j=1}^D \mathcal{F}_{v_j}\mathbb{E}_{\mathcal{T}_L, \mathbb{F}}\Big[S_j\Big] \right]\\ &=&  \sum_{k=0}^{\infty} p \left(1-p \right)^k \left[(k+1) p +  \frac{k D}{ \sum_{i=1}^{D}  \mathcal{F}_{v_i}} \mathbb{E}_{\mathcal{T}_L, \mathbb{F}}\Big[S\Big] \right],
	\end{eqnarray*}
	where in the last equality we use the 
    notation 
    \begin{equation}\label{eq_Sthetatilde}
		S= \frac{1}{D} \sum_{i=1}^D  \mathcal{F}_{v_i} S_i, 
	\end{equation} 
     where we assume that the different $S_i$ are coupled via the same graphical construction on the same underlying tree but have different initial conditions (but in fact only the marginal distributions matter, as we are only interested in the expectation of $S$). 
    Calculating  the series in the  previous display explicitly, we obtain
	\begin{equation}\label{eq_Stheta}
		\mathbb{E}_{\mathcal{T}_L, \mathbb{F}}\left[\widetilde{S}_L \right] = 1 +  \lambda D \mathcal{F}_{\rho} \mathbb{E}_{\mathcal{T}_L, \mathbb{F}}\left[S\right].
	\end{equation} 
	
	Now,  we shall find a {\color{black}recursive} upper bound for $\mathbb{E}_{\mathcal{T}_L, \mathbb{F}}[\widetilde{S}_L] $ 
    by establishing a relationship between $ \mathbb{E}_{\mathcal{T}_L, \mathbb{F}}[S] $ and the stationary distribution of the contact process $\textbf{CP}_{\rho}(\mathcal{T}_L)$. Denote by  $S_{L-1}^{i}$ the first time when \nat{$\textbf{CP}_{\rho}(\mathcal{T}_{L-1}^{+, v_i}; \mathbf{1}_{v_i})$} 
    becomes \textbf{0} on $\mathcal{T}_{L-1}^{v_i}$. For 
\nat{$\textbf{CP}_\rho(\mathcal{T}^{+,v_i}_{L-1})$}, the rate of leaving the state \textbf{0} and the expected return time to \textbf{0} are given by
	\begin{equation*}
		q_i(\textbf{0}) := \lambda \mathcal{F}_{\rho} \mathcal{F}_{v_i}  \quad  \quad \text{and}\quad \quad  m_i({\textbf{0}}) := \big(q_i(\textbf{0})\big)^{-1} + \mathbb{E}_{\mathcal{T}_{L-1}^{+,v_i}, \mathbb{F}}\Big[S_{L-1}^{i}\Big].
	\end{equation*}
 Similarly, for the process  $\textbf{CP}_{\rho}(\mathcal{T}_L)$  we also find these two quantities 
\[\begin{split}
		q(\textbf{0}) := \lambda \mathcal{F}_{\rho}  \sum_{i=1}^{D}  \mathcal{F}_{v_i}\quad \text{and}\quad m({\textbf{0}}) := \frac{1}{q(\textbf{0})} + \sum_{i=1}^{D} \frac{\lambda \mathcal{F}_\rho \mathcal{F}_{v_i}}{q(\textbf{0})} \mathbb{E}_{\mathcal{T}_L, \mathbb{F}}\Big[S_i\Big],
	\end{split}
\]
where we recall that  $S_i$ is the first time when   {$\textbf{CP}_{\rho}(\mathcal{T}_L; \mathbf{1}_{v_i})$ reaches the completely recovered state \textbf{0}} 
on $\cup_{i=1}^D \mathcal{T}_{L-1}^{v_i}$.
Let $\widetilde{\nu}_{\mathcal{T}_L}$ and $\nu_{\mathcal{T}_{v_i}}$ be the stationary distribution of $\textbf{CP}_{\rho}(\mathcal{T}_L)$ and $\textbf{CP}_{\rho}(\mathcal{T}_{L-1}^{+,v_i}),$ respectively. Therefore, we have 
	\begin{equation*}
		\widetilde{\nu}_{\mathcal{T}_L}(\textbf{0}) = \frac{1}{q(\textbf{0})m(\textbf{0})}= \frac{1}{1+\lambda D \mathcal{F}_{\rho} \mathbb{E}_{\mathcal{T}_L, \mathbb{F}}\left[S \right]}.
	\end{equation*}
Note {\color{black}also in $\textbf{CP}_{\rho}(\mathcal{T}_L)$, the process restricted to each subtree $\mathcal{T}_{L-1}^{v_i}$ evolves independently, }so that $\widetilde{\nu}_{\mathcal{T}_L}: = \otimes_{j=1}^D \nu_{\mathcal{T}_{v_j}}$, which yields
 	\begin{equation}\label{eq_nuthetanucruz}
		\widetilde{\nu}_{\mathcal{T}_L}(\textbf{0})=\frac{1}{1+\lambda D \mathcal{F}_{\rho} \mathbb{E}_{\mathcal{T}_L, \mathbb{F}}\left[S \right]} = \prod_{i=1}^{D}\frac{1}{1+\lambda \mathcal{F}_{\rho} \mathcal{F}_{v_i}\mathbb{E}_{\mathcal{T}_{L-1}^{+,v_i}, \mathbb{F}}\Big[S^{i}_{L-1} \Big]}.
	\end{equation}
Using this fact, we obtain from \eqref{eq_Stheta} that
	\begin{eqnarray*}
		\mathbb{E}_{\mathcal{T}_L, \mathbb{F}}\left[\widetilde{S}_L \right] = 1 +  \lambda D \mathcal{F}_{\rho} \mathbb{E}_{\mathcal{T}_L, \mathbb{F}}\left[S\right] =   \prod_{i=1}^{D} \left(1+\lambda \mathcal{F}_{\rho} \mathcal{F}_{v_i}\mathbb{E}_{\mathcal{T}_{L-1}^{+,v_i}, \mathbb{F}}\Big[S_{L-1}^{i} \Big]\right).
	\end{eqnarray*}
Denote by $\widetilde S_{L-1}^{i}$ the first time when
\nat{$\widetilde{\textbf{CP}}_{\rho}(\mathcal{T}^{+,v_i}_{L-1}; \mathbf{1}_{v_i})$} reaches \textbf{0}, {\color{black}where this process is defined as $\textbf{CP}_{\rho}(\mathcal{T}^{+,v_i}_{L-1}; \mathbf{1}_{v_i})$ but $v_i$ is only allowed to recover if the subtree of all its children are all healthy.} By using the monotonicity of the contact process, we have $S_{L-1}^{i} \leq \widetilde S_{L-1}^{i}$. Now, taking expectations conditionally only on $D$ and $\mathcal{F}_\rho$ and using that  $\widetilde S_{L-1}^{i}$ does not depend on $\mathcal{F}_\rho$ by construction, we get 
\[\begin{split} \mathbb{E}\left[\widetilde{S}_L \mid D,\mathcal{F}_{\rho}\right] 
	&\leq  \mathbb{E}\left[\prod_{i=1}^{D}\left(1+ \lambda \mathcal{F}_{\rho} \mathbb{E}\Big[\mathcal{F}_{v_i}\widetilde S_{L-1}^{i}| \mathcal{T}_{v_i},  \mathbb{F}(\mathcal{T}_{v_i})\Big]\right)\Big|\Big.D, \mathcal{F}_{\rho}\right] \\ &= \prod_{i=1}^{D}\left(1+ \lambda \mathcal{F}_{\rho}  \mathbb{E}\Big[\mathcal{F}_{v_i}\widetilde S_{L-1}^{i}\Big]\right).
\end{split}\]

Moreover, we have
\[\begin{split} \mathbb{E}\left[\widetilde{S}_L \mid D,\mathcal{F}_{\rho}\right] = \prod_{i=1}^{D}\left(1+ \lambda \mathcal{F}_{\rho}  \mathbb{E}\Big[\mathcal{F}_{v_i
}\widetilde S_{L-1}^i\Big]\right) =  \left(1+\lambda \mathcal{F}_{\rho} \mathbb{E}\left[\mathcal{F}_{\rho}\widetilde S_{L-1} \right]\right)^D.
\end{split}\]
Thus taking expectations over $D$ and $\mathcal{F}_\rho$, we obtain
\begin{equation}\label{eq_FStheta}
	\mathbb{E}\big[\mathcal{F}_\rho \widetilde S_L \big] \leq  \mathbb{E}\left[\mathcal{F}_\rho \left(1+\lambda \mathcal{F}_{\rho} \mathbb{E}\left[\mathcal{F}_{\rho}\widetilde S_{L-1}\right]\right)^D\right].
	\end{equation}
 
	Now, we shall apply an inductive argument over $L$ to bound $ \mathbb{E}\left[\mathcal{F}_\rho \widetilde S_L \right]$. We assume that  $\mathbb{E}[\mathcal{F}_{\rho}(1+c\mathcal{F}_\rho)^{\xi}]=M<\infty$ for some $c>0$, and define the following constants 
	$$ K=  M\max\left\{\frac{1}{c}, 1\right\},\quad \lambda_0 = \frac{1}{2K}.$$
	For the base case $L=0$, note that $\widetilde S_0$ is an exponential random variable with parameter $1$
 that is independent of $\cF_\rho$. Therefore, 
	$$ \mathbb{E}\big[\mathcal{F}_\rho \widetilde S_0\big]
= \mathbb{E} \big[\mathcal{F}\big] 	 \mathbb{E} \big[\widetilde S_0\big] \leq M,$$
	so that the base case holds.
	Next, we assume $\mathbb{E}[\mathcal{F}_\rho \widetilde S_{L-1}]\leq M$. Then  we have for any $\lambda \le \lambda_0$ that
	\begin{equation*}
		\frac{\lambda \mathbb{E}\left[\mathcal{F}_{\rho}\widetilde S_{L-1} \right]}{c} \leq \frac{ \lambda M  }{c}  \leq  \frac{M}{2 c K} = \frac{1}{2 c \max\{1/c,1\}}  \leq 1.
	\end{equation*}
	Hence,  since $g(x)=\mathcal{F}_{\rho}(1+x\mathcal{F}_\rho)^{D}$ for $x>0$ is an increasing function, 
 we have 
 for any $L$ and all $\lambda\leq \lambda_0$ 
	\begin{equation}
		\begin{split}
		\mathbb{E}\big[\mathcal{F}_\rho \widetilde S_L \big] &\leq   \mathbb{E}\Big[g\big(\lambda \mathbb{E}[\mathcal{F}_\rho \widetilde S_{L-1}]\big)\Big]\leq  \mathbb{E} [ g(c) ] =   \mathbb{E}\big[\mathcal{F}_{\rho}(1+c\mathcal{F}_\rho)^{D}\big] = M,
		\end{split}
	\end{equation}
 where the first inequality follows from \eqref{eq_FStheta}. Moreover,  taking into account the monotonicity of the contact process, we have $S_L  \leq \widetilde{S}_L $ which yields, for any $L$ and all $\lambda\leq \lambda_0$,
 \begin{equation}\label{eq:boundS}
		\begin{split}
		\mathbb{E} \big[\mathcal{F}_\rho R_L\big] \le \mathbb{E}\big[\mathcal{F}_\rho S_L \big] \leq 	\mathbb{E}\big[\mathcal{F}_\rho \widetilde S_L \big] \leq  M,
		\end{split}
 \end{equation}
as required.
\end{proof}

We are now ready to move to the proof of the main theorem of this section.

\begin{proof}[Proof of Theorem \ref{teo_exptail}]
We consider $\lambda_0$ as defined in Lemma \ref{lem_stationdistr}. Let $\lambda\in (0,\lambda_0]$. Let $R$ denote the first time when $ \text{\textbf{CP}}(\mathcal{T}; \mathbf{1}_{\rho})$ reaches state \textbf{0}. Then by \nat{the} monotone convergence theorem and Lemmas \ref{lem_stationdistr}, we obtain 
   \[\mathbb{E}[\mathcal{F}_{\rho}R] = \mathbb{E}\left[\lim_{L\to \infty} \mathcal{F}_{\rho}R_L\right] = \lim_{L\to \infty} \mathbb{E}[\mathcal{F}_{\rho}R_L] \le M.\]
   Now since $\mathcal{F}_{\rho}R$ is a non-negative random variable, we have that   $\mathcal{F}_{\rho}R<\infty$ almost surely.
   Therefore 
   \[\mathbb{P}\big(X_t \not= \textbf{0}\ \text{for all}\ t \geq 0\big)  = \mathbb{P}(R=\infty) = \mathbb{P}(\mathcal{F}_{\rho} R=\infty)=0,\]
   where in the second equality we have used that $\mathcal{F}_\rho > 0$ almost surely. We conclude that 
for all $\lambda \leq \lambda_0$
the process $(X_t)\sim \text{\textbf{CP}}(\mathcal{T}; \mathbf{1}_{\rho})$ dies out almost surely.
\end{proof}

\section{Finite stars}\label{sec_finitestars}

In this section, we show some results for the inhomogeneous contact
process on stars which will be used in the proof of Theorem \ref{teo_noexptail} in the next section. Although any finite graph is eventually trapped in the state of zero infection, the stars are able to maintain the infection for a long time. Here we will show that, if  the root of the star is initially infected, the star will keep a large number of infected leaves for a long time with probability very close to 1. 

Some of our results in this section are inspired by results obtained by Huang and Durrett in \cite[Section 2]{huang2019exponential}. However, we had to adapt their arguments to take advantage of the fact that we can have a large fitness value associated to the root of the star. Throughout this section, we assume that the random variable $\mathcal{F}$ takes values in $[\vartheta,\infty)$ for some $\vartheta>0$.

We start this section by proving a lower bound for the probability to transfer the infection from one vertex to another in a graph consisting of a single path conditionally on the associated fitness values.

\begin{lemma}\label{lem_chain}
	Let $r$ be an arbitrary non-negative integer and $f\geq \vartheta$ a  real number. Let  $\mathcal{C}_r$ be a graph consisting of a single path of length $r$ on the vertices $v_0,\dots, v_r$ with  associated fitness values $\mathbb{F}(\mathcal{C}_r):=\{\mathcal{F}_{v_0},\dots, \mathcal{F}_{v_r}\}$. Consider $(X_t)\sim {\normalfont{\text{\textbf{CP}}(\mathcal{C}_r; \mathbf{1}_{v_0})}}$ the contact process on $\mathcal{C}_r$ where $v_0$ is initially infected. Then there exists a constant $\gamma>0$ such that
	\begin{equation*}
		\mathbb{P}_{\mathbb{F}}\bigg(v_r \in  \bigcup_{s\leq 2r} X_{s} \bigg)=	\mathbb{P}\bigg(v_r \in \bigcup_{s\leq 2r} X_{s}\ |\ \mathbb{F}(\mathcal{C}_r), v_{0}\in X_0 \bigg)\geq  \big(1-e^{-\gamma r}\big) \prod_{i=1}^{r} \frac{\lambda \mathcal{F}_{v_{i-1}}\mathcal{F}_{v_i}}{1+\lambda \mathcal{F}_{v_{i-1}}\mathcal{F}_{v_i}}.
	\end{equation*}
Moreover,  on the event that $\{ \mathcal{F}_{v_0}\geq f, \mathcal{F}_{v_r}\geq f\} $ we have that
		\[\begin{split}
			\mathbb{P}\Big(v_r \in \bigcup_{s\leq 2r}  X_{s}\ | & \ \mathbb{F}(\mathcal{C}_r), v_{0}\in X_0 \Big) \geq \big(1-e^{-\gamma r})  C_{\lambda, f} \left(\frac{\lambda\vartheta^2}{\lambda\vartheta^2+1}\right)^r,
		\end{split}\]
		where
		\begin{equation}\label{eq_ctelf}
			C_{\lambda,f}:= \left(\frac{\lambda\vartheta^2+1 }{\lambda\vartheta^2}\right)^{2}\left(\frac{\lambda f   \vartheta }{1 + \lambda f   \vartheta}\right)^2.
	\end{equation}
\end{lemma}

\begin{proof}
	We emphasize here that the notation $	\mathbb{P}_{\mathbb{F}}\big(\cdot \big)$ corresponds to the conditional probability on the fitness and also that we start with $v_0$ initially infected. Let $r$ be an arbitrary non-negative integer.  First, we need to establish some appropriate notation. We define the sequence of times $(s_i)_{i\geq 0}$ by setting  $s_0=0$ and for $i\in  \{1, \dots, r\}$ defining
	\begin{equation*}
		s_i:= \inf \{s\geq s_{i-1}: \ v_{i-1}\ \text{recovers or  infects}\ v_i \mbox{ at time } s\}.
	\end{equation*}
Denote  $T=\sum_{i=1}^{r}t_i$\   where \ $t_i:=s_i - s_{i-1}$.  Also denote  the events $$
B_i:=
\{ v_{i-1} \text{ infects } v_i \text{ before recovering}\}
\quad \quad  
\text{and}\quad \quad B=\bigcap_{i=1}^rB_i.$$ 
We begin by noting that from the definition of the event $B$ and the definition of $T$ we can obtain the following lower bound  
	\begin{eqnarray}\label{eq_proba_lemmapath}
		\mathbb{P}_{\mathbb{F}}\bigg(v_r \in \bigcup_{s\leq 2r} X_{s} \bigg)  &\geq &  \mathbb{P}_{\mathbb{F}}\big(B\cap \left\{ T \leq 2r\right\}\big)= \mathbb{P}_{\mathbb{F}}\big(\left. T\leq 2r\, \right| \, B\big)\mathbb{P}_{\mathbb{F}}(B).
	\end{eqnarray}
	Now, we shall establish lower bounds for the two  probabilities on the right-hand side above. 
	Conditioning on the fitness and also on the event $\{v_{i-1}\in X_{s_{i-1}}\}$, we know that the probability that $v_{i-1}$ infects $v_i$ before recovering
	is given by
	\begin{equation*}
		\mathbb{P}_{\{\mathcal{F}_{v_{i-1}}, \mathcal{F}_{v_i}\}}\big(B_i \ |\  v_{i-1}\in X_{s_{i-1}} \big)=  \frac{\lambda \mathcal{F}_{v_{i-1}}\mathcal{F}_{v_i}}{1+\lambda \mathcal{F}_{v_{i-1}}\mathcal{F}_{v_i}}.
	\end{equation*}
where here we denote by $\mathbb{P}_{\{\mathcal{F}_{v_{i-1}}, \mathcal{F}_{v_i}\}}(\cdot)$  the conditional probability $\mathbb{P}(\cdot \ |\{\mathcal{F}_{v_{i-1}}, \mathcal{F}_{v_i}\} )$.
 This implies that
	\begin{equation}\label{eq_lemmapath2}
		\mathbb{P}_{\mathbb{F}}(B) = \prod_{i=1}^{r} 	\mathbb{P}_{\{\mathcal{F}_{v_{i-1}}, \mathcal{F}_{v_i}\}}\big(B_i \ |\  v_{i-1}\in X_{s_{i-1}} \big) = \prod_{i=1}^{r} \frac{\lambda \mathcal{F}_{v_{i-1}}\mathcal{F}_{v_i}}{1+\lambda \mathcal{F}_{v_{i-1}}\mathcal{F}_{v_i}}.
	\end{equation}

On the other hand, by an application of  Markov's inequality and by the definition of $B$ and $T$, for  any $\theta >0$, 
\[ \begin{aligned} \mathbb{P}_\mathbb{F}\big(T\geq 2r\ | \ B \big) &=  \mathbb{P}_\mathbb{F}\big(\left. e^{\theta T} \geq e^{2\theta r}\ \right|  B \big)\nonumber \leq e^{-2\theta r}\mathbb{E}_{ \mathbb{F}}\big[e^{\theta T}\ | \ B\big] \\ 
	&\leq e^{-2\theta r}\prod_{i=1}^{r}\mathbb{E}_{\{\mathcal{F}_{v_{i-1}}, \mathcal{F}_{v_i}\}}\Big[e^{\theta t_i}\ \Big|\Big. \ B_i\cap\{v_{i-1}\in X_{s_{i-1}}\}\Big]
\end{aligned} \]
 By conditioning on the event $B_i\cap\{v_{i-1}\in X_{s_{i-1}}\}$ we obtain that $t_i$ has an exponential distribution with parameter $(1+\lambda \mathcal{F}_{v_{i-1}} \mathcal{F}_{v_i})$. Therefore, we can couple $t_i$ with a random variable $\tau_i$ with an exponential distribution with parameter $1$ such that $t_i \leq \tau_i$ almost surely. 
Then, following the standard argument for a large deviation bound, we obtain that 
\[ \mathbb{P}_\mathbb{F}\big(T\geq 2r\ | \ B \big) \leq e^{-2 r \theta^* + r \log \phi(\theta^*)}, \]
where $\phi(\theta^*) = \E[ e^{\theta^* \tau_i}]$. 
Now, note that 
\[ \lim\limits_{\theta^* \to 0}\frac{\log(\phi_i(\theta^*))}{\theta^*} = \E[ \tau_i] = 1. \]
Therefore, by choosing $\theta^* > 0$ small enough, we can deduce that there exists $\gamma > 0$ such that  
 	\begin{eqnarray}\label{eq_lemmapath3}
		\mathbb{P}_\mathbb{F}\big(T\geq 2r\ | \ B \big) &\leq & e^{-r \gamma}.
	\end{eqnarray}
Plugging \eqref{eq_lemmapath2} and \eqref{eq_lemmapath3} back into \eqref{eq_proba_lemmapath}, we now see that 
	\begin{equation*}
		\mathbb{P}_{\mathbb{F}}\bigg(v_r \in \bigcup_{s\leq 2r} X_{s} \bigg)\geq  \big(1-e^{-\gamma  r}\big) \prod_{i=1}^{r} \frac{\lambda \mathcal{F}_{v_{i-1}}\mathcal{F}_{v_i}}{1+\lambda \mathcal{F}_{v_{i-1}}\mathcal{F}_{v_i}}.
	\end{equation*}

For the second part, we fix  $f\geq \vartheta$. Then, on the event $\{\mathcal{F}_{v_0} \geq f,\mathcal{F}_{v_r}\geq f\}$
and using that $\cF_{v_i} \geq \vartheta$, we obtain 
	\begin{equation*}
\prod_{i=1}^{r} \frac{\lambda \mathcal{F}_{v_{i-1}}\mathcal{F}_{v_i}}{1+\lambda \mathcal{F}_{v_{i-1}}\mathcal{F}_{v_i}} \geq \left(\frac{\lambda\vartheta^2}{1+\lambda\vartheta^2}\right)^{r-2}\left(\frac{\lambda f   \vartheta}{1+\lambda f   \vartheta}\right)^2,
	\end{equation*}
which yields the desired result.
\end{proof}

Let $G_k$ be a star of size $k$, that is, $G_k$ consists of a root $\rho$ and $k$ other vertices each connected only to $\rho$,  denoted by $v_1,\dots, v_k$.  Let $(X_t)\sim {\normalfont{\text{\textbf{CP}}(G_k; \mathbf{1}_{\rho})}}$ denote the inhomogeneous contact process on $G_k$ with the root initially infected.   Define $(X_t^0)$ to be the modification of $(X_t)$  in such a way that the fitness values for  the leaves satisfy $\mathcal{F}_{v_1}=\dots = \mathcal{F}_{v_k}=\vartheta$ and the fitness of the root is $\mathcal{F}_{\rho}=f \geq \vartheta$. Note that the contact process  $(X_t^0)$  may be considered as a contact process without fitness and with rate parameter $\tilde{\lambda}:=\lambda f\vartheta$. Now, on the event $\{\mathcal{F}_\rho \geq f\}$  and taking into account that the random variable $\mathcal{F}$ takes values in $[\vartheta,\infty)$, we have 
\begin{equation*}
	\lambda  f \vartheta \leq \lambda \mathcal{F}_{v_j} \mathcal{F}_\rho \quad \quad \text{for every} \quad  j=1,\dots, k.
\end{equation*}
Thus by monotonicity of the contact process in the rate parameter, as discussed in Section~\ref{sec_properties}, we have $X_t^0\subset X_t$ on the event $\{\mathcal{F}_\rho \geq f\}$.
Moreover, on a star the dynamics of $(X_t^0)$ are the same as that of a standard contact process, 
but where the infection rate is $\lambda f\vartheta$. In particular, we
can use some of the results in Section 2 of Huang and Durrett \cite{huang2018contact} (see also  \cite[ Lemma 2.2]{chatterjee2009contact}) describing 
the persistence of the infection on a star. Some of the results can be taken over directly, however, 
others need to be adapted so that we can make full use of the fact that 
we have the extra flexibility of making $f$ large enough.

When considering the process $(X_t^0)$, write the state of the star as $(m,n)$ where $m$ is the number of infected leaves and $n=0$ or $1$ if the center is healthy or infected, respectively.
Throughout, we will write $\Lambda_t^0 \subset X_t^0$ for the set of infected leaves. Also, we will write $\mathbb{P}_{(m,n)}$ if we are conditioning on $X_0^0 = (m,n)$.

As in~\cite{huang2018contact, chatterjee2009contact}, we will reduce  the dynamics to
a one-dimensional chain, by concentrating on  the first coordinate (i.e.\ we will count the number of infected leaves). 
As a first step, we will ignore the times when the centre is not infected and as a second step we will stop 
the dynamics when we reach a certain level $L$ of infected leaves. 
We can then define a suitable 
time-homogeneous Markov chain that lower bounds the number of infected leaves (running on a clock ignoring 
times when the centre is not infected).

To deal with the number of leaves that recover while the root is not infected, we note that
when the state is $(m,0)$ for some $m>0$, the next event will occur after an exponential time with mean $1/(m\lambda f \vartheta + m)$. The probability that the root is reinfected first is  $\lambda f \vartheta / (\lambda f \vartheta + 1)$.  Denote by $\mathfrak{N}$ the number of infected leaves  that will recover while the center is healthy. Thus $\mathfrak{N}$ has a shifted geometric distribution with success probability $\lambda f \vartheta / (\lambda f \vartheta +1)$, i.e. 
\begin{equation}\label{eq_rvN}
		\mathbb{P}\big(\mathfrak{N}=j\big) = \left(\frac{1}{\lambda f \vartheta+1}\right)^{j}\left(\frac{\lambda f \vartheta}{\lambda f \vartheta +1}\right), \quad j\geq 0.
	\end{equation} 

Fix $\lambda>0, k\geq 1$ and $f\geq \vartheta$ and set a cut-off level 
\begin{equation}\label{eq:defL}
L=\left\lceil\frac{\lambda f \vartheta k }{1+2\lambda f \vartheta}\right\rceil.
\end{equation}
If we  modify the chain so that the infection rate is 0 when the number of infected leaves is $\geq L$, 
then we can couple the number of infected leaves to a process $(Y_t)_{t \geq 0}$
with the following dynamics
\begin{equation}\label{eq:Y}
\begin{array}{lcc}
	\text{jump} & & \text{at rate} \\ 
	Y_t \to Y_t -1, &    & L \\
	Y_t \to \min\{Y_t +1, L\}, &   & \lambda f  \vartheta (k-L) \\
	Y_t \to Y_t -\mathfrak{N},  &    & 1,
\end{array}	
\end{equation}
so that the process $(Y_t)_{t\ge 0}$ stays below the number of infected leaves (ignoring times when the centre is not infected) as long as the original process has not hit the state $(0,0)$ yet. 
For convenience, we do not stop the process after hitting a state below $0$ and instead we are careful to apply 
the coupling only up to the hitting of $(0,0)$.

	The following lemma shows that $|\Lambda_t^0|$ hits level $L$ before the process dies out with high probability. 
	Also, we show that the first time $(Y_t)$ hits $L$ has small expectation. 
Our result is similar to \cite[Lemma 2.5]{huang2018contact} and \cite[Lemma 2.3]{chatterjee2009contact}, however we need to adapt their arguments to give useful estimates also for large fitness.

To formalize these statements, denote for the original chain for any $A \geq 0$,
\[ T_A = \inf\{t\geq 0 \, :\, |\Lambda_t^0| \geq A \}, \quad 
T_{0,0} = \inf\{ t \geq 0 \, : \, X_t^0 = (0,0) \}. \]
Moreover, for $(Y_t)$ define for $A \geq 0$, 
 $$  T_A^{Y} = \inf\{t\geq 0: Y_t\geq A \} \qquad \text{and}\qquad R_0^{Y} = \inf\{t\geq T_1^{Y}: Y_t \leq 0\} .$$

In the following lemma, we will also consider the embedded discrete time process $(Z_n)$ of $(Y_t)$ obtained by looking at $Y_t$ only at its jump times. This process has the property  that for either $f$ or $k$ large enough  $((1+\lambda f  \vartheta /2)^{{-Z_n}})$ is a supermartingale while ${Z_n}\in (0,L).$ The proof follows from similar arguments as those used in \cite[Lemma 2.1]{huang2018contact}. 
\nat{
\begin{lemma}\label{lem:supermtg}
Let $\lambda>0$, $f\ge \vartheta$ and $k\geq 7$. For either $f$ or $k$ large enough  $((1+\lambda f  \vartheta /2)^{{-Z_n}})$ is a supermartingale while ${Z_n}\in (0,L).$  
\end{lemma}}

\begin{proof}
Define $e^\theta= (1+\lambda f\vartheta/2)^{-1}$ and assume that $Z_n\in (0,L)$. After {\color{black}the same straightforward calculations as in \cite[Lemma 2.1]{huang2018contact} (but with infection rate $\lambda f$),} we obtain with $D= L+\lambda f\vartheta (k-L) +1 \ge 0$, for $y>0$, {\color{black}
\begin{equation*}
\begin{aligned}
    \mathbb{E}[e^{\theta Z_{n+1}} - e^{\theta Z_n} \mid Z_n=y] & = \frac{e^{\theta y}}{D} \left((e^{-\theta}-1)L + (e^\theta -1)\lambda f\vartheta (k-L)  + \frac{e^{-\theta}-1}{1+\lambda f\vartheta - e^{-\theta}}\right) \\
    & = \frac{e^{\theta y}}{D} \left(\frac{\lambda f \vartheta}{2}L - \frac{\lambda f \vartheta}{2+\lambda f \vartheta }\lambda f\vartheta (k-L)  + 1\right) ,
    \end{aligned}
\end{equation*}
where we used the definition of $\theta$.
It remains to show that the term in the bracket on the right hand side is negative. We write using the definition of $L$ above 
\[ \begin{aligned} \frac{\lambda f \vartheta}{2}L & - \frac{\lambda f \vartheta}{2+\lambda f \vartheta }\lambda f\vartheta (k-L)  + 1 \\
& = \lambda f \vartheta\Big( \frac{2+3\lambda f \vartheta}{2(2+\lambda f \vartheta)}L - \frac{\lambda f \vartheta}{2 + \lambda f \vartheta} k \Big) + 1 \\
& \leq \lambda f \vartheta\bigg( \frac{2+3\lambda f \vartheta}{2(2+\lambda f \vartheta)}\Big(\frac{\lambda f \vartheta k}{1+2\lambda f \vartheta}+1\Big)  - \frac{\lambda f \vartheta}{2 + \lambda f \vartheta} k \bigg) + 1 \\
& = - \frac{(\lambda f \vartheta)^3}{2(2+ \lambda f \vartheta)(1+2\lambda f \vartheta)} k 
+ \lambda f \vartheta \frac{2+3\lambda f \vartheta}{2(2+\lambda f \vartheta)}  + 1 .
\end{aligned}\]
Therefore, by calculating the asymptotics, we can deduce that the right hand side is negative, if 
either $k \geq 7$ and $f$ is large or if $k$ is large. }
\end{proof}

\begin{lemma}\label{lem_cotasTKT0}
\nat{Let $\lambda>0$, {\color{black}$k\geq 7$} and $f\geq   \vartheta$.} Consider the stochastic process $(Y_t)$ defined in \eqref{eq:Y} and the contact process $(X_t^0)$. Then for either $f$   or $k$ large enough, 
	\begin{equation*}
		\mathbb{P}_{(0,1)}\big(T_L > T_{0,0}\big) \leq c_1(\lambda, f, \vartheta, k) \qquad \text{and} \qquad  \mathbb{E}_{0}\left[T_L^{Y} \right] \leq 8 \max\left\{ 1, \frac{1}{(\lambda f \vartheta)^2}\right\},
	\end{equation*}
where
    \begin{equation}\label{eq:def_c_1}
        c_1(\lambda, f, \vartheta, k) = \frac{9}{\lambda f   \vartheta k^{1/3}}+(1+\lambda \nat{f} \vartheta /2)^{-\frac 13 k^{1/3}} + e^{- \frac 18 (\lambda f \vartheta)^2 k^{1/3}}.
    \end{equation}
\end{lemma}

\begin{remark}
Note that \cite[Lemma 2.5]{huang2018contact} also states a result for the conditional expected value of $T_L$. 
In their proof the authors ignore times when the root is not infected, so that their result for the expected value is really only for $T_L^Y$. 
However, the  estimate on the probability is only true for the original process and not for $Y_t$.
We fix this omission by bounding the times when the root is not infected in the next lemma.
\end{remark}

\begin{proof}
		Let $\lambda>0$, $k\geq 1$ and $f\geq   \vartheta$ and define
		 $$K=
         \left\lceil \frac{\lambda   \vartheta {\color{black}f} k^{1/3}}{1+2\lambda  \vartheta {\color{black}f}} \right\rceil$$
	Recall the definition of  the constant $L$ given in \eqref{eq:defL} and note  that $K\leq L$. 
 We begin by observing that 
\begin{equation*}
		\mathbb{P}_{(0,1)}\big(T_K < T_{0,0}\big) \geq \prod_{j=0}^{K-1}\frac{(k-j)\lambda f  \vartheta}{1+(k-j)\lambda f   \vartheta + j},
\end{equation*}
where the term in the product corresponds to the probability that $|\Lambda_t^0|$ jumps upwards $K$ times before either the root or one of the leaves recovers.
From the latter inequality and using \cite[Lemma 3.4.3]{durret2019book}, we have
\[\begin{split}
		\mathbb{P}_{(0,1)}\big(T_K > T_{0,0}\big) &\leq \prod_{j=0}^{K-1}1 - \prod_{j=0}^{K-1}\frac{(k-j)\lambda f   \vartheta}{1+(k-j)\lambda f   \vartheta + j} \leq \sum_{j=0}^{K-1} \frac{1+j}{1+(k-j)\lambda f   \vartheta+ j}\\ &\leq \sum_{j=0}^{K-1} \frac{1+j}{(k-j)\lambda f   \vartheta}  \leq \frac{K^2}{(k-K)\lambda f   \vartheta},
	\end{split}\]
where in the last inequality we have used that $\{(1+j)/(k-j), j=0, \dots, K-1\}$ is  increasing in $j$. 

Note that as the function $x \mapsto \frac{x k^{1/3}}{1+2x}$ is increasing in $x$, we can upper bound the function by its limit as $x \rightarrow \infty$, so that  
\[ K \leq \Big\lceil \frac{k^{1/3}}{2} \Big\rceil \leq \frac 12 k^{1/3} + 1\leq \frac{3}{2} k^{1/3} , \]
as $k \geq 1$. Moreover, note that $\frac{3}{2} k^{1/3} \leq \frac{3}{4} k$ for all $k \geq 7$. Using these two estimates we have that 
 \[\begin{split}
 	\mathbb{P}_{(0,1)}\big(T_K > T_{0,0}\big)
& \leq \frac{K^2}{(k-K)\lambda f \nat{\vartheta}}
\leq c \frac{1}{\lambda f \vartheta k^{1/3}} ,
\end{split}\]
for $c = 4( \frac{3}{2})^2 =9$.

Now, we use the jump process $(Z_n)$ and \nat{Lemma \ref{lem:supermtg}.} We assume that $k$ or $f$ is large, so that $(1+\lambda f \vartheta/2)^{-Z_n}$ is a supermartingale while $Z_n \in (0,L)$.
We denote by $T_L^Z$ and $R_0^Z$ the analogous hitting time of a level above $L$ and the return time to a level below $0$  for the process $(Z_n)$, respectively. Note that, since $(Z_n)$ is obtained from $(Y_t)$ by observing the process only at its jump times, we see that $\{ T_L^{Y} > R_0^{Y} \} = \{ T_L^{Z} > R_0^{Z} \}$. By an application of the Optional Stopping Theorem with the bounded stopping time $\tau \wedge n$ where $\tau= R_0^Z \wedge T_L^Z$, we get 
	$$ \mathbb{E}_{K}\Big[\big(1+\lambda f   \vartheta/2\big)^{-Z_{\tau\wedge n}}\Big] \leq \mathbb{E}_{K}\Big[\big(1+\lambda f   \vartheta/2\big)^{-Z_{0}}\Big].$$
	By letting $n\to \infty$, we deduce  
		$$ \mathbb{E}_{K}\Big[\big(1+\lambda f   \vartheta/2\big)^{-Z_{\tau}}\Big] \leq \big(1+\lambda f/2\big)^{-K},$$
 	which implies
	\begin{equation*}
		\mathbb{P}_{K}\big(R_{0}^Z < T_{L}^Z\big) + \Big(1-\mathbb{P}_{K}\big(R_{0}^Z < T_{L}^Z\big)\Big) \big(1+\lambda f   \vartheta / 2\big)^{-L} \leq \big(1+\lambda f   \vartheta / 2\big)^{-K}.
	\end{equation*}
Discarding the second term on the left-hand side, it follows that  
	\begin{equation}\label{eq_bound_L}
			\mathbb{P}_{K}\big(R_{0}^Z < T_{L}^Z\big)\leq \big(1+\lambda f  \vartheta / 2\big)^{-K}.
	\end{equation}
 To go back to the unmodified process, we use that $\{ T_{0,0} < T_L \} \subset \{ R_0^Z < T_L^Z \}$, so 
that we can apply~\eqref{eq_bound_L}. Combined with the strong Markov property we obtain 
	\begin{eqnarray*}
		\mathbb{P}_{(0,1)}\big(T_L > T_{0,0} \big) 
		&\leq &	\mathbb{P}_{(0,1)}\big(T_K > T_{0,0}\big) + 	\mathbb{P}_{(0,1)}\big(T_L > T_{0,0} \, |\,  T_K <  T_{0,0}\big) \mathbb{P}_{(0,1)}\big(T_K < T_{0,0} \big)\\ 
		&\leq &	
       \mathbb{P}_{(0,1)}\big(T_K > T_{0,0}\big) + 	\mathbb{P}_{(K,1)}\big(T_L > T_{0,0}\big)\\
		&\leq &	
        \mathbb{P}_{(0,1)}\big(T_K > T_{0,0}\big) + 	\mathbb{P}_{K}\big(R_0^Z < T_L^Z\big).
	\end{eqnarray*}
Then, combining with the above estimates, we have the first claim of the lemma, that is, 
	\begin{eqnarray*}
		\mathbb{P}_{(0,1)}\big(T_L > T_{0,0}\big) &\leq& \nat{c}\frac{1}{\lambda f  \vartheta k^{1/3}}  + \big(1+ \lambda f   \vartheta / 2\big)^{-K}.
	\end{eqnarray*} 

    The first claim of the lemma now follows, if we note that 
    \[ \big(1+\lambda f   \vartheta / 2\big)^{-K} \leq
    (1+\lambda \nat{f} \vartheta /2)^{- \frac{\lambda f \vartheta}{1+ \nat{2}\lambda  f \vartheta}k^{1/3} } \leq (1+\lambda \nat{f} \vartheta /2)^{-\frac 13 k^{1/3}} + e^{- \frac 38 (\lambda f \vartheta)^2 k^{1/3}}, \]
    where we used that if $\lambda f \vartheta \geq 1$, as $x \mapsto \frac{x}{1+2x}$ is increasing, 
    \[\frac{\lambda f \vartheta}{1+ \nat{2}\lambda  f \vartheta}k^{1/3} \geq \frac 13 k^{1/3}.  \]
    Moreover, if $\lambda f \vartheta \leq 1$, then using that $\log(1 + x) \geq \frac 34 x$ for $x \leq 1/2$,
    \[ \frac{\lambda f \vartheta}{1+ \nat{2}\lambda  f \vartheta}k^{1/3} \log(1+ \lambda f \vartheta /2 ) \geq \frac{1}{8}(\lambda f \vartheta)^2 k^{1/3} .\]
   
	To deal with the second claim of the lemma, 
	recall the random variable $\mathfrak{N}$ defined in \eqref{eq_rvN} with $\mathbb{E}_{0} [\mathfrak{N}] = (\lambda f   \vartheta)^{-1}$.  
	From the definition of the process, we can calculate the drift of $Y_t$ before time $T_L^Y$ as
	\begin{equation*}
		\mu_{\lambda, f}= -L + \lambda f   \vartheta (k-L) - \frac{1}{\lambda f   \vartheta},
		\end{equation*}
so that  $(Y_t - \mu_{\lambda, f} t)$  is a martingale when stopped at $T_L^Y$. 

An application of the  Optimal Stopping Theorem with 
the bounded stopping time $T_L^Y  \wedge t$ tells us that
		\begin{equation*}
			\mathbb{E}_{0}\big[Y_{T_L^Y  \wedge t}\big] -  \mu_{\lambda, f} \mathbb{E}_{0}\big[T_{L}^Y  \wedge t\big] = \mathbb{E}_{0}[Y_0] = 0.
		\end{equation*}
		Furthermore, if we can show that $\mu_{\lambda,f} > 0$, 
        since $\mathbb{E}_{0}[Y_{\nat{T_L^Y}  \wedge t}]  \leq L$, we get
        the upper bound 
		\begin{equation*}
        \begin{split}
					\mathbb{E}_{0}\big [T_{L}^Y  \wedge t\big] &\leq \frac{L}{\mu_{\lambda, f}}.
        \end{split}
		\end{equation*}
Moreover, by taking $t \to \infty$, the same bound holds for $\mathbb{E}_{0}\big [T_{L}^Y ]$.

Thus, it remains to upper bound  $L / \mu_{\lambda,f}$.
 Write $x=\lambda f \vartheta$ and consider
 \begin{equation*}
 \begin{split}
\mu_{\lambda, f} &= xk -L(1+x) -\frac{1}{x}\ge xk-\left(\frac{xk}{1+2x}+1\right)(1+x)-\frac{1}{x}\\ &= \frac{x^2 k}{1+2x}-(1+x) - \frac{1}{x} 
 \end{split}
 \end{equation*}
Therefore, 
\begin{equation}\label{eq:2602-1} \mu_{\lambda,f} \ge \frac{1}{2} \frac{x^2 k}{1+2x} 
\end{equation}
holds if 
\[   \frac{x^2 k}{1+2x}\ge 2\Big(1+x + \frac{1}{x}\Big)
\quad \iff\quad  k\ge 2\Big(2+ \frac{3}{x} + \frac{3}{x^2} + \frac{1}{x^3}\Big) . 
\]
This condition can be guaranteed by either taking 
$k \geq 5$  and then $f$ large enough (so that $x = \lambda f \vartheta$ is large)  so that the right hand side is smaller than $5$. 
Alternatively, the condition can be achieved by taking $k \geq 18 \max\{ \frac{1}{x},1\} = 18\max\{ \frac{1}{\lambda f \vartheta},1\}$. 
 
In either case, we can deduce from~\eqref{eq:2602-1} that 
\[
     \frac{L}{\mu_{\lambda, f}} \le \frac{\frac{xk}{1+2x}+1}{\frac{1}{2}\frac{x^2 k}{1+2x}}= 2 \Big(\frac{1}{x}+\frac{2}{xk}+\frac{1}{x^2k}\Big) \leq 2 \Big(\frac{3}{x} + \frac{1}{x^2k}\Big) \nat{\le 8 \max\{ 1, x^{-2}\} } ,
\]
which gives the result claimed. 
\end{proof}

By combining the two results of the previous lemma with an estimate on the time that the root of the star is not infected, 
we can show that the process reaches $L$ infected leaves before time $fk^{2/3}$ with high probability for either $f$ or $k$ large.

\begin{lemma}\label{lem:TL1}
	Let $G_k$ be a star with leaves $v_1,\dots, v_k$ and root $\rho$, {\color{black}where $k\ge 7$}. Consider the  contact process $(X_t^0)\sim {\normalfont{\text{\textbf{CP}}(G_k; \mathbf{1}_\rho)}}$.  Let $\lambda >0$ be fixed. Then, for either $f$  or $k$ large enough,  
	$$ 	\mathbb{P}_{(0,1)}\big(T_L>fk^{2/3}\big)\leq \nat{c_1(\lambda, f, \vartheta, k) + \frac{32 \max\{ 1, (\lambda f \vartheta)^{-3}\}}{f k^{2/3}}} 
    + e^{-\frac{\widehat{c}_1}{4} \nat{ \min \{ 1, \lambda f \vartheta   \}} fk^{2/3}} +  {\color{black}2\,}e^{-\widehat{c}_2\lambda f^2\vartheta k^{2/3}},$$
	where {\color{black}$c_1(\lambda,f,\vartheta,k)$ is defined in~\eqref{eq:def_c_1}} and $\widehat{c}_1$ and $\widehat{c}_2$ 
    are positive constants {\color{black}not depending on any of the parameters}. 
\end{lemma}

\begin{proof}
Fix $\lambda>0$. Recall the process $(Y_t)_{t\ge 0}$ given in \eqref{eq:Y} and its corresponding stopping times $T_L^Y$ and  $R_0^Y$. We begin by letting the stochastic process $(\Gamma_t)_{t \geq 0}$ 
 denote  the number of infected leaves, where we ignore times when the root is not infected.
Let us now introduce the stopping time $T_L^\Gamma = \inf\{t>0: \Gamma_t \geq L \}$.  
 Then, by definition we have $T_L \geq T_L^\Gamma$ a.s.. 
Moreover, on the event that $T_L < T_{0,0}$, we have that 
$T_L^\Gamma  \leq T_L^Y$.  

Let us denote $g(k)=fk^{2/3}$ and choose \nat{$\delta>0$} as
\[\nat{\delta} :=  \min \{ 1, \lambda f \vartheta   \} / 4 . \] 
Observe that 
\begin{equation}\label{eq:TL}
	\begin{split}
\mathbb{P}_{(0,1)}\big(T_L> g(k)\big) 
 &\leq  \mathbb{P}_{(0,1)}\big(T_L^\Gamma \geq\delta g(k)\big) +  \mathbb{P}_{(0,1)}\big(T_L>g(k), \ T_L^\Gamma <\delta g(k)\big).
	\end{split}
\end{equation}
Our objective is now to find upper bounds for both probabilities on the right-hand side above. First, we deal with the first  of these probabilities. Note that
  $$ \begin{aligned} \mathbb{P}_{(0,1)}\big(T_L^\Gamma \geq \delta g(k)\big) & \leq 
\mathbb{P}_{(0,1)}\big( T_L > T_{0,0}\big) + 
\mathbb{P}_{(0,1)}\big(T_L^\Gamma \geq \delta g(k) ,  T_L < T_{0,0}\big) \\
& \leq 
c_1(\lambda, f, \vartheta, k) + \mathbb{P}_0( T_L^Y \geq \delta g(k)\big), 
\end{aligned} $$ 
where we used the first part of Lemma \ref{lem_cotasTKT0}. By the second part of the same lemma and Markov's inequality, we also have that
\[ \mathbb{P}_0( T_L^Y \geq\delta g(k)\big) \leq \, \frac{\E_0[T_L^Y]}{\delta g(k)}  \leq \nat{\frac{8 \max\{ 1, (\lambda f \vartheta)^{-2}\}}{\delta g(k)}}.
\]
Combining, we obtain 
\[  \mathbb{P}_{(0,1)}\big(T_L^\Gamma \geq\delta g(k)\big)
\leq \nat{c_1(\lambda, f, \vartheta, k) + \frac{8 \max\{ 1, (\lambda f \vartheta)^{-2}\}}{\delta g(k)}}.
\]
To estimate the second term in~\eqref{eq:TL},
let  $\ell_t$ denote the amount of time until $t$ that the process $(X_t^0)$ spends in the states when the root is not infected, i.e.\
$$\ell_t=\big|\big\{s \leq t: X_s^0=(j,0),\ \text{for some} \  j\in \{0,1, \dots, k \}\big\}\big|,$$
where $|\cdot|$ denotes the Lebesgue measure.
Note that as $T_L= T_L^{\Gamma} + \ell_{T_L}$,  it follows that
\begin{equation}
	 \mathbb{P}_{(0,1)}\big(T_L>g(k), \ T_L^\Gamma <\delta g(k)\big)\leq \mathbb{P}_{(0,1)}\big(T_L^\Gamma<\delta g(k), \ \ell_{T_L}>(1-\delta) g(k)\big).
\end{equation}
Now, we recall the random variable $\mathfrak{N}$ defined in \eqref{eq_rvN}. Denote by $m_t$ the number of times that the process $(\Gamma_t)$ makes a downward jump with the same distribution as $\mathfrak{N}$ until time~$t$.  Note that if  the current state is $(i,0)$ with $i\geq 1$, then the time until either the next jump down happens or when the root is re-infected is given by  an exponential random variable with distribution \textbf{Exp}$(i(1+\lambda f   \vartheta))$. Such a random variable is stochastically dominated by  another random variable that has distribution \textbf{Exp}$(1+\lambda f   \vartheta)$. Hence, each time period when the root is healthy can be dominated by 
$$\sum_{i=1}^{\mathfrak{N}+1}E_i,$$
where for each $i=1, \dots, \mathfrak{N}+1$  the random variables $E_i$ have independent \textbf{Exp}$(1+\lambda f  \vartheta)$ distributions.  Thus on the event, $\{T_L^\Gamma <\delta g(k)\}$, we deduce that 
$$\ell_{T_L} \leq \sum_{j=1}^{m_{T_L^\Gamma}}\sum_{i=1}^{\mathfrak{N}_j+1}E_i^{(j)} \leq\sum_{j=1}^{m_{\delta g(k)}}\sum_{i=1}^{\mathfrak{N}_j+1}E_i^{(j)},$$
where the random variables  $E_i^{(j)}$ have \textbf{Exp}$(1+\lambda f   \vartheta)$  distributions and $\mathfrak{N}_j$ are independent random variables with the same distribution as $\mathfrak{N}$. Further, as $\mathfrak{N_j}$ has a shifted geometric distribution, straightforward computations show that for each $j$,
$$\bar T_j:=\sum_{i=1}^{\mathfrak{N}_j+1}E_i^{(j)}$$  has an exponential distribution with parameter $\lambda f \vartheta$. Denote $a:=(1-\delta)(2\delta)^{-1}$.
It follows by a large deviation upper bound {\color{black}(or Chernoff bound)} that 
\begin{equation*}
	\begin{split}
\mathbb{P}_{(0,1)}&\big(T_L^\Gamma <\delta g(k), \ \ell_{T_L}>(1-\delta)g(k), \ m_{\delta g(k)} \leq 2\delta g(k)\big)\\ & \hspace{4cm}\leq  \mathbb{P}_{(0,1)} \Bigg(\sum_{j=1}^{\lfloor 2\delta g(k)\rfloor}
\bar T_j
> (1-\delta)g(k)\Bigg) \leq  e^{-\lfloor 2\delta g(k) \rfloor I(a)},
	\end{split}
\end{equation*}
using that $\mathbb{E}[ \bar T_j]  = 1/{(\lambda f \vartheta)} \leq (1-\delta)/(2 \delta)$ by the definition of $\delta$
and where $I(a)$ is the rate function of an Exp$(\lambda f \vartheta)$-random variable, i.e.
\begin{equation*}
    I(a) :=\sup_{t\ge 0}\left(ta - \log \mathbb{E}\big[e^{t\bar T_j}\big]\right) = a\lambda f \vartheta -1 -\log (a\lambda f \theta).
\end{equation*}
Note that if $\lambda f \vartheta \geq 1$, then $\delta = 1/4$ and $a = 3/2$. Moreover, in this case using that $x - 1 - \log x \geq x / 20$ for $x \geq 3/2$, we have that
\[ I(a) = \frac 32 \lambda f \theta  - 1 - \log \Big(\frac 32  \lambda f \theta\Big) \geq \frac{1}{20} \, \lambda f  \theta. \]
Conversely, if $\lambda f \vartheta \leq 1$, then $\delta = \lambda f \vartheta / 4$
and 
\[ a \lambda f \vartheta \geq 2 - \frac{1}{2} \lambda f \vartheta \geq \frac{3}{2} , \]
so that by the same bound as before $I(a) \geq a\lambda f\vartheta/20 \ge 3/40$. Thus,
\[
e^{-\lfloor 2\delta g(k)\rfloor I(a)}
\;\le\;
\begin{cases}
e^{\lambda f \vartheta/20}e^{-\lambda f\vartheta g(k)/40},
& \text{if } \lambda f \vartheta \ge 1, \\[2ex]
e^{3/40} e^{-3\lambda f\vartheta g(k)/80},
& \text{if } \lambda f \vartheta \le 1 .
\end{cases}
\]
If $g(k) = fk^{2/3}\ge 4$, then $e^{\lambda f \vartheta/20}e^{-\lambda f\vartheta g(k)/40}\le e^{-\lambda f\vartheta fk^{2/3}/80}$. Combining both cases,  and using that $e^{3/30}\leq 2$,
we deduce that for $k$ or $f$ sufficiently large
\begin{equation*}
	\begin{split}
\mathbb{P}_{(0,1)}&\big(T_L^\Gamma <\delta g(k), \ \ell_{T_L}>(1-\delta)g(k), \ m_{\delta g(k)} < 2\delta g(k)\big) \leq e^{-\lfloor 2\delta g(k)\rfloor I(a)}\leq  2e^{-\lambda f\vartheta g(k)/80}.
	\end{split}
\end{equation*}

In addition, note that the random variable $m_{\delta g(k)}$ has a Poisson distribution with parameter $\delta g(k)$. Similarly as before, using a large deviation bound {\color{black}(or Chernoff bound)} we have
\begin{equation*}
    \mathbb{P}_{(0,1)} \big(m_{\delta g(k)}> 2 \delta g(k)\big) 
    \le e^{-I(a)},
\end{equation*}
where 
$a=2\delta g(k)$ and
\begin{equation*}
    I(a)= \sup_{t\ge 0}\left(ta - \delta g(k)(e^t-1)\right) = \delta g(k) \left(2\log2-1\right).
\end{equation*}
Therefore, with $\widehat{c}_1=2\log2 -1$ and $\widehat{c}_2 = 1/80$,
\begin{equation*}
	\begin{split}
\mathbb{P}_{(0,1)}&\big(T_L^\Gamma < \delta g(k), \ \ell_{T_L}^\Gamma> (1-\delta) g(k)\big)   \\& \leq  \mathbb{P}(m_{\delta g(k)}>2\delta g(k))+\mathbb{P}_{(0,1)}\big(T_L^\Gamma<\delta g(k), \ \ell_{T_L}>(1-\delta)g(k), m_{\delta g(k)} < 2\delta g(k)\big) \\ & \leq  e^{-\widehat{c}_1\delta g(k)} + {\color{black}2\,} e^{-\widehat{c}_2\lambda f\vartheta g(k)}.
\end{split}
\end{equation*}
Substituting the estimates back into~\eqref{eq:TL}, we obtain
\begin{equation*}
	\begin{split}
		\mathbb{P}_{(0,1)}\big(T_L>g(k)\big) &\leq  \mathbb{P}_{(0,1)}\big(T_L^\Gamma >\delta g(k)\big) +  \mathbb{P}_{(0,1)}\big(T_L>g(k), \ T_L^\Gamma <\delta g(k)\big) \\& \leq 
        \nat{c_1(\lambda, f, \vartheta, k) + \frac{8 \max\{ 1, (\lambda f \vartheta)^{-2}\}}{\delta g(k)} + }
        e^{-\widehat{c}_1\delta g(k)} +  {\color{black}2\,}e^{-\widehat{c}_2\lambda f\vartheta g(k)}.
	\end{split}
\end{equation*}
This concludes the proof.
\end{proof}

The following lemma, proved by Huang and Durrett in \cite{huang2018contact}, will be useful for our next result.

\begin{lemma}[{\cite[Lemma 2.4]{huang2018contact}}]\label{lem_star}
	Let  $k$  be  an arbitrary non-negative integer and $f\geq   \vartheta$ a real number. Let $G_k$ be a star with leaves $v_1,\dots,v_k$ and root $\rho$. Consider  the contact process $(X_t^0)\sim {\normalfont{\text{\textbf{CP}}(G_k; \mathbf{1}_{\{\rho, v_1, \dots, v_L\}})}}$ where $\rho$ and $L=\lceil\lambda f   \vartheta k/(1+2\lambda f   \vartheta)\rceil$ leaves are initially infected. Then, for  {any  $\epsilon \in (0,1/2)$,} 
	\begin{equation}\label{eq:bound_star}
		\mathbb{P}_{(L,1)}\left( \inf_{0\leq t\leq S}|X_t^0| \leq \epsilon L \right) \leq (3 + \lambda f   \vartheta) (1+\lambda f   \vartheta  /2)^{-\epsilon L},
	\end{equation}
	where 
	\begin{equation}\label{eq_timeS}
		S
  = \frac{1}{4k}(1+\lambda f   \vartheta/2)^{L(1-2\epsilon)-1}.
	\end{equation}
\end{lemma}

\newcommand{\epszero}{\frac{5}{16}}

{\color{black}
\begin{remark}[Choice of $\epsilon$]
We note that when we want to apply Lemma~\ref{lem_star}, then we would like the probability on the right hand side of~\eqref{eq:bound_star} to go to zero and at the same time, we want $S$ to go to $\infty$. 
In the case when $f \rightarrow \infty$ and $k$ is fixed,
 we see that this requires that we choose $\epsilon$ such that
\begin{equation}\label{eq:choice_eps} \epsilon L > 1 \quad \mbox{and}\quad L(1-2\epsilon) > 1 . \end{equation}
This is only possible if $1/L < (L-1)/(2L)$ or in other words if $L > 3$. Note as $f \rightarrow \infty$, we have that 
$L \rightarrow \lceil k/2 \rceil$. Thus, the smallest value of $k$ we can choose is $k = \kzero$, so that the smallest $L$ is given by $L= 4$. 
In that case, we can check that~\eqref{eq:choice_eps} holds for $\epsilon = \frac{5}{16}$ , which will be our choice from now on.

Note also that in the case when $k \rightarrow \infty$ \nat{and $f$ is fixed}, both conditions hold automatically.
\end{remark}}

The next lemma tell us that if either $f$ or $k$ are large enough, then beginning with only  vertex $\rho$ infected at time 0, the number of infected leaves during the time interval $[fk^{2/3},S]$ is at least $\epsilon L$ with high probability.

\begin{lemma}\label{lem_infecleaves}
	Let  ${\color{black}k\ge \kzero}$  be  an integer and $f\geq   \vartheta$ a real number. {Let $G_k$ be a star of size $k$ with root $\rho$. Consider $(X_t)\sim {\normalfont{\text{\textbf{CP}}(G_k;\mathbf{1}_{\rho})}}$   the inhomogeneous contact process on $G_k$ \nat{and $\Lambda_t \subset X_t$ it set of infected leaves}.  Fix $\lambda >0$, then for $\epsilon = {\color{black}\epszero}$, and for either $f$  or $k$ large enough it holds that on the event $\{\mathcal{F_\rho} \geq f\}$}
	\begin{equation*}
		\mathbb{P}_{(0,1)}\Big(\inf_{fk^{2/3} \leq t \leq S} |\Lambda_t| \leq \epsilon L \ \Big|\Big. \ { \mathbb{F} }  \Big) \leq  R(f,k,\lambda),
	\end{equation*}
	where 
\begin{equation}\label{eq_Rf} \begin{aligned}
		R(f,k,\lambda) & =
      c_1(\lambda, f, \vartheta, k) + \frac{32}{f k^{2/3}} \max\{ 1, (\lambda f \vartheta)^{-3}\}
        + e^{-\widehat{c}_1 fk^{2/3}} +  {\color{black}2\,}e^{-\widehat{c}_2\lambda f^2\vartheta k^{2/3}}  \\
          & \qquad {\color{black} + \min\bigg\{ \widehat{C}_1 (2 + \lambda f \vartheta)^{-1/8} ,(3 + \lambda f \vartheta) (1 + \lambda \vartheta^2/2)^{ - \epszero \frac{\lambda\vartheta^2}{1+2\lambda \vartheta^2} k } \bigg\} }
    \end{aligned}
    \end{equation}
	and where \nat{$\widehat{c}_1, \widehat{c}_2, \widehat{C}_1 $ are positive constants not depending on any of the parameters {\color{black}and $c_1$ is defined in~\eqref{eq:def_c_1}}.}
\end{lemma}

\begin{proof}
	Fix $\lambda >0$. Here we use the notation $\mathbb{P} ( \cdot ):= \mathbb{P} (\ \cdot \ | \ \mathbb{F})$. We begin by noting that on the event $\{\mathcal{F}_\rho\geq f\}$ and by monotonicity we have
	$$ \left\{\inf_{fk^{2/3} \leq t \leq S} |\Lambda_t| \leq \epsilon L\right\} \subset \left\{\inf_{fk^{2/3} \leq t \leq S} |\Lambda_t^0| \leq \epsilon L\right\}.$$
	Then it is enough to prove the  estimate for the process $(\Lambda_t^0)_{t\ge 0}$. By the strong Markov property applied at $T_L$ on the event that 
	$T_L < fk^{2/3}$ and using that at $T_L$ the root is necessarily infected, we note that
	\begin{equation}\label{eq:probaT} 
		\begin{aligned}
			\mathbb{P}_{(0,1)}\left(\inf_{fk^{2/3} \leq t \leq S} |\Lambda_t^0| \leq \epsilon L \right) &= \mathbb{P}_{(0,1)}\left(\inf_{fk^{2/3} \leq t \leq S} |\Lambda_t^0| \leq \epsilon L, \  T_L \geq  fk^{2/3} \right)  \\  &\hspace{1.5cm}+ \     \mathbb{P}_{(0,1)}\left(\inf_{fk^{2/3} \leq t \leq S} |\Lambda_t^0| \leq \epsilon L,  \ T_L < fk^{2/3} \right)  \\ &\leq \mathbb{P}_{(0,1)}\big(T_L \geq  fk^{2/3}\big)    +    \mathbb{P}_{(L,1)}\left(\inf_{0 \leq t \leq S} |\Lambda_t^0| \leq \epsilon L\right).
		\end{aligned}
	\end{equation}
	
Now, appealing to Lemma \ref{lem_star}, we have  for any  $\epsilon \in (0,1/2)$ and for either $f$ or $k$ large enough,  
	\begin{equation*}
	\mathbb{P}_{(L,1)}\left( \inf_{0\leq t\leq S}|\Lambda_t^0| \leq \epsilon L \right) \leq
    (3 + \lambda f   \vartheta) (1+\lambda f   \vartheta  /2)^{-\epsilon L}.
\end{equation*}
{\color{black}Now, if $k\geq \kzero$ is fixed we can choose $f$ sufficiently large so that $L \geq \frac{9}{10}\lceil \frac{k}{2} \rceil$ and $\lambda f \vartheta \geq 1$. Therefore, 
\[(3 + \lambda f   \vartheta) (1+\lambda f   \vartheta  /2)^{-\epsilon L} \leq 
(3 + \lambda f   \vartheta) (1+\lambda f   \vartheta  /2)^{-\epszero \frac{9}{10}4 } 
\leq \hat{C_1} (2 + \lambda f \vartheta)^{-1/8} .\]
 for some absolute constant $\widehat{C}_1$. 
For $k$ large, we use that $x \mapsto \frac{x}{1+x}$ is increasing, so that
\[ (3 + \lambda f   \vartheta) (1+\lambda f   \vartheta  /2)^{-\epsilon L} \leq 
(3 + \lambda f \vartheta) (1 + \lambda \vartheta^2/2)^{ - \epszero \frac{\lambda\vartheta^2}{1+2\lambda \vartheta^2} k }.
\]
Plugging these two bounds back into \eqref{eq:probaT} and using Lemma \ref{lem:TL1}, we get the desired result.
}
\end{proof}

One more lemma will be needed for the next section, so we record it now. {The result describes the behaviour of the contact process on a graph consisting of a star with a single path joined to one of its leaves. We give a lower bound for the probability that the vertex on the path that is furthest from the root will be infected, if we start with the root of the star infected. This is a similar result to \cite[Lemma 3.2]{huang2018contact}.}

\begin{lemma}\label{lem_starchain}
	Let $r\geq 1$, ${\color{black}k\ge \kzero}$ be integers, and $f\geq   \vartheta$ a real number. Let $G_k$ be the star of size $k$ {with root $\rho$ and leaves $v_1, \dots, v_k$}, to which has been added a single {path} of length $r$ of descendants of some child $v_i$ of $\rho$. Denote  by $\mathcal{C}_{r}$ {the path with vertices  $u_1,\dots, u_r$ with $u_1=v_i$ and associated fitness values  $\{\mathcal{F}_{u_1},\dots, \mathcal{F}_{u_r}\}$}. Consider $(X_t)\sim {\normalfont{\text{\textbf{CP}}(G_k\cup \mathcal{C}_r; \mathbf{1}_{\rho})}}$  the inhomogeneous contact process on $G_k\cup \mathcal{C}_r$ where $\rho$  is initially infected. Then, {\color{black}for $S$ as in~\eqref{eq_timeS} with $\epsilon = \epszero$ and} for any $\tilde{S}\in [S/2, S]$ and either $f$ or $k$ large enough, 
 on the event $\{ \mathcal{F}_{ u_{r}} \geq f, \mathcal{F}_{\rho} \geq f\}$ that
	\begin{equation*}		\mathbb{P}_{(0,1)}\big({u_{r}}\notin X_{s}\ \text{for all}\ s \in[0,\tilde{S}] \  \big|\big. \ \mathbb{F} \big) \leq \left(1- \widehat{C}_{\lambda,f}\widehat{\lambda}^r\right)^{ S/4(2r+1) -1} + R(f,k,\lambda), 
	\end{equation*}
where
	\begin{equation}\label{eq_lhat}
\widehat{C}_{\lambda,f}=(1-e^{-\lambda \vartheta fL})(1-e^{-\gamma})C_{\lambda,f}\quad \quad \text{and}\quad \quad 	\widehat{\lambda}=\frac{\lambda   \vartheta^2}{\lambda   \vartheta^2+1}.
\end{equation}
  The terms  $S$, $C_{\lambda,f}$ and $R(f,k,\lambda)$ were defined in \eqref{eq_timeS}, \eqref{eq_ctelf} and \eqref{eq_Rf}, respectively.
\end{lemma}

\begin{proof}
	Let $m := \lfloor \tilde{S}(2r+1)^{-1} \rfloor$ and $m_0 := \lceil fk^{2/3}(2r+1)^{-1} \rceil$. Since either $f$ or $k$ are large enough we can assume that $fk^{2/3}\le S/4$ and w.l.o.g.\  $m_0\le m$.   Here we use the notation $\mathbb{P} ( \cdot ):= \mathbb{P} (\ \cdot \ | \ \mathbb{F})$  
 and we add the subscript $\mathbb{P}_{(0,1)}$ when only the root is initially infected. Begin by noting that 
	\[\begin{split}
		\mathbb{P}_{(0,1)} \left({u_{r}}\notin X_{s}\ \text{for all}\ s \in[0,\tilde{S}]\right)  & \leq 	\mathbb{P}_{(0,1)}(\mathcal{B}^c)  \\ &\quad + \mathbb{P}_{(0,1)}\left( {u_{r}}\notin X_{s}\ \text{for all}\ s \in[0,m(2r+1)], \mathcal{B}  \right),
	\end{split}\]
	where
 \[ \mathcal{B}:=\left\{\inf_{fk^{2/3}\leq s\leq S}|\Lambda_s| \geq \epsilon L\right\}\]
 and $|\Lambda_t|$ is the number of infected leaves of $\rho$ at time $t$. From Lemma \ref{lem_infecleaves} we know for $f$ and $k$ large enough that
	\begin{equation*}
\mathbb{P}_{(0,1)}\left(\mathcal{B}^c\right) \leq  R(f,k,\lambda),
	\end{equation*}
	where $R(f,k,\lambda)$ was defined in \eqref{eq_Rf}. The proof is thus complete as soon as we can show that for $f$ and $r$ sufficiently large 
	\begin{equation}\label{eq_leavesinfec}
		\begin{split}
		\mathbb{P}_{(0,1)}\left( {u_{r}}\notin X_{s}\ \text{for all}\ s \in[0,m(2r+1)],\ \mathcal{B} \right)  \leq  \big(1- \widehat{C}_{\lambda,f}\widehat{\lambda}^r\big)^m.
		\end{split}
	\end{equation}
 Define the sequence of times $t_0=0$ and $t_i=(2r+1)i$ for $i \in \{1,\dots, m\}$. For fixed $i \in \{0,\dots, m-1\}$, let us define  $\tau_i := \inf\{u\ge t_i: \rho \in X_u \} $,  
 \[B_{\tau_i} := \left\{\inf_{s \in [t_i, \tau_i\wedge t_{i+1}]} |\Lambda_{s}|\geq \epsilon L\right\}\qquad \text{and}\qquad A_i:=\{u_{r}\notin X_{s}\ \text{for all}\ s \in[t_{i},t_{i+1}]\}.\]
 Thus, with this notation we have
 \begin{eqnarray*}
 		 \mathbb{P}_{(0,1)}\left({u_{r}}\notin X_{s}\ \text{for all}\ s \in[0,m(2r+1)] ,  \mathcal{B}  \right) 
    \leq \mathbb{P}_{(0,1)}\left(\bigcap_{i=m_0}^{m-1}A_i, \bigcap_{i=m_0}^{m-1}B_{\tau_i}  \right). 
 \end{eqnarray*}
Since $\mathcal{B}_{\tau_i}\in \mathcal{F}_{\tau_i}$, $\tau_i$ is $\mathcal{F}_{\tau_i}$-measurable and $\mathcal{F}_{t_i}\subset \mathcal{F}_{\tau_i}$, we deduce 
\begin{equation*}
\begin{split}
    &\mathbb{P}\left(u_{r}\in X_{s}\  \text{for some}\ s \in[t_i,t_{i+1}] \ \big|\big. \  B_{\tau_i}, \ \mathcal{F}_{t_i} \right) 
    \\ & \hspace{3cm} \ge \mathbb{P} \left(\{u_{r}\in X_{s}\ \text{for some}\ s \in[\tau_i,t_{i+1}]\} \cap \{\tau_i\leq t_{i}+1\} \ \big|\big. \  B_{\tau_i}, \ \mathcal{F}_{t_i}\right)
    \\ &   \hspace{3cm} \geq 
   \E \Big( \mathbf{1}_{\{ \tau_i \leq t_i+1 \} } \mathbb{P} \big(u_{r}\in X_{s}\ \text{for some}\ s \in[\tau_i,\tau_i + 2r ] \ \big|\big. \   \mathcal{F}_{\tau_i}\big) \, \Big|\Big. \, B_{\tau_i} ,\mathcal{F}_{t_i}\Big).
   \end{split}
\end{equation*}
Now note that by an application of the strong Markov property
and monotonicity, we obtain 
	\[\begin{split}
 \mathbb{P}(u_{r}\in  X_{s}\ \text{for some}\ s \in[\tau_i,\tau_{i} + 2r] \mid \mathcal{F}_{\tau_i})&\ge   \mathbb{P}(u_{r}\in  X_{s}\ \text{for some}\ s \in[0,2r]\mid  X_0 = \{ \rho \})  \\
 & \geq \big(1-e^{-\gamma }\big)  C_{\lambda, f} \widehat{\lambda}^r, 
\end{split}\]
where we used Lemma \ref{lem_chain} together with the fact that $1-e^{-\gamma} \leq 1-e^{-\gamma r} $ holds for $r\geq 1$ in the last step.
In addition, under $B_{\tau_i}$, the probability that $\tau_i \leq t_i+1$ is bounded from below by the probability that one of the $L$ exponential clocks with rate $\lambda \vartheta f$ attached to the leaves rings before time 1, i.e.
 \[\mathbb{P}(\tau_i\leq t_i+1\mid  B_{\tau_i}, \mathcal{F}_{t_i})\ge 1-e^{-\lambda \vartheta f L}.\]
Combining the estimates, we have that
\[ \mathbb{P}\left(u_{r}\notin X_{s}\  \text{for all}\ s \in[t_i,t_{i+1}] \ \big|\big. \  B_{\tau_i}, \ \mathcal{F}_{t_i} \right) 
\leq 1- (1-e^{-\lambda \vartheta fL})(1-e^{-\gamma})C_{\lambda, f}\widehat{\lambda}^r.
\]
Therefore, by the tower property
\[\begin{aligned}
\mathbb{P}_{(0,1)}\bigg( & \bigcap_{i=m_0}^{m-1}A_i \cap B_{\tau_i} \bigg)\\
& = \mathbb{E}_{(0,1)} \left( \mathbf{1}_{\bigcap_{i=m_0}^{m-2}  A_{i}\cap B_{\tau_i}}  \mathbb{P}\left(u_{r}\not\in X_{s}\  \text{for all}\ s \in[t_{m-1},t_{m}] , B_{\tau_{m-1}} \big|\big. \ \mathcal{F}_{t_{m-1}} \right)\right)
\\ & \leq \mathbb{E}_{(0,1)}\Big(
\mathbf{1}_{\bigcap_{i=m_0}^{m-2}  A_{i}\cap B_{\tau_i}}
\mathbb{P}\left(u_{r}\not\in X_{s}\  \text{for all}\ s \in[t_{m-1},t_{m}] \big|\big. \ B_{\tau_{m-1}} ,\mathcal{F}_{t_{m-1}} \right) \Big).
\end{aligned}\]
Using the previous bound and iterating, 
\begin{eqnarray*} \mathbb{P}_{(0,1)}\left({u_{r}}\notin X_{s}\ \text{for all}\ s \in[0,m(2r+1)]  ,  \mathcal{B}  \right)  & \leq &  \prod_{i=m_0}^{m-1} \left(1- (1-e^{-\lambda \vartheta fL})(1-e^{-\gamma})C_{\lambda, f}\widehat{\lambda}^r\right)\\  & \leq  &  \left(1- \widehat C_{\lambda, f}\widehat{\lambda}^r\right)^{m-m_0}.
 \end{eqnarray*}
Thus the desired result follows since $m-m_0\ge S/4(2r+1)-1$.
\end{proof}

\section{Proof of Theorem \ref{teo_noexptail}}\label{sec_prooftheononexptails}

The proof of the theorem follows a similar overall structure as those 
in {\cite[Theorem 3.2]{pemantle1992contact} and afterwards in \cite[Theorem 1.4]{huang2018contact}.} However, the presence of fitness turns out to lead to  significant changes throughout the whole proof.	
The proof is long and we break it up into several lemmas.  The general idea is to push the infection to stars  in a suitably chosen  generation and then  bring the infection back to the root. 

In the first lemma, we estimate from below the mean of the number of stars in a given generation. In order to do so, we first introduce some notation. 
Note that we  can choose $\vartheta > 0$ small enough such that
\begin{equation}\label{eq:muvartheta}
   \mu_\vartheta:=\mathbb{E}\big[\xi \mathbf{1}_{\{\mathcal{F} \geq \vartheta\}}\big] > 1. 
\end{equation}
By monotonicity, we can also assume 
that this $\vartheta$ is such that condition~\eqref{eq:offspringfitness} also holds for this choice.
Throughout the remainder of the section, we fix this $\vartheta>0$.  Let $f\ge \vartheta$ be a real number and ${\color{black}k\ge \kzero}$ a natural number. 
\nat{In what follows, we will assume either $f$ or 
$k$ sufficiently large, {\color{black}so that the results of the previous section hold.}} Let 
\begin{equation}\label{eq:root}
 A_{f,k}^{\rho} :=\{|\{v \in V_1:\  \mathcal{F}_v\ge \vartheta \quad \text{and}\quad \xi_v\ge 1\}|=k \quad \text{and}\quad \mathcal{F}_\rho \ge f\},  
\end{equation}
where $|\cdot |$ denotes the cardinality of the set and $V_1$ the set of vertices in generation one as defined in \eqref{eq_Vr}. 

Let $r\ge 1$ \nat{be} a natural number. Denote by $\mathcal{G}_r$ and \nat{$\mathcal{A}_{f,k}^r$} the set of \textit{good vertices} and \nat{$k$-}\textit{stars} in generation $r$ in $(\mathcal{T},\mathbb{F}(\mathcal{T}))$, i.e.
	\begin{equation*}
		\mathcal{G}_{r}:= \big\{v\in V_r: \ \mathcal{F}_{w}\geq \vartheta \quad \text{for all}\quad  w \quad  \text{ancestors of} \quad v\big\} 
	\end{equation*}
 and
 \begin{equation*}
		\nat{\mathcal{A}_{f,k}^r}:= \big\{v\in\mathcal{G}_r: \ \mathcal{F}_{v}\geq f \quad \text{and}\quad   \xi_v^{\vartheta}=k\big\},
	\end{equation*}
where $\xi_v^{\vartheta} = \sum_{w \in c(v)} \mathbf{1}_{\{ \mathcal{F}_w \geq \vartheta\}}$, for $c(v)$ the set of children of $v$. We note for later that 
\nat{the conditional distribution of $\xi_v^{\vartheta}$ given the tree up to generation $r$} is the same as that of 
\begin{equation}\label{eq:xivartheta}
     \xi^\vartheta := \sum_{i=1}^{\xi} I_i,
\end{equation}
with $I_1, I_2, \dots $ Bernoulli random variables with success probability $\mathbb{P}(\mathcal{F}\geq \vartheta)$.

\begin{lemma}\label{lem:zr}
Let $r\ge 2$. Denote by $Z_r$ the number of \nat{$k$-}stars in generation $r$ in $(\mathcal{T},\mathbb{F}(\mathcal{T}))$. Then 
\begin{equation*}
    \mathbb{E}\big[Z_r \mid A_{f,k}^{\rho}\big] \ge k \mu_\vartheta^{r-2}\mathbb{P}\left(\xi^\vartheta=k,\ \mathcal{F}\geq f\right),
\end{equation*}
where $\mu_\vartheta$ and $\xi^\vartheta$ are defined in \eqref{eq:muvartheta} and \eqref{eq:xivartheta}, respectively.
\end{lemma}

\begin{proof}
We will assume from now on that $\mathbb{P}$ refers to the conditional probability  measure given $A_{f,k}^{\rho}$. Denote by $\mathbb{T}_{\le r}$ and $\mathbb{F}_{\le r}$ the $\sigma$-algebras \nat{generated} by the BGW tree and the fitness values up to generation $r$, respectively. We also denote by $c(v)$ the set of children of vertex $v\in V(\mathcal{T})$. We begin by noting that
    \[\begin{split}
        \mathbb{E}\big[Z_r \mid \mathbb{T}_{\le r+1},\,  \mathbb{F}_{\le r}\big] &=  \sum_{v\in \mathcal{G}_r} \mathbf{1}_{\{\mathcal{F}_v\ge f\} }   
        \mathbb{E}\big[\mathbf{1}_{\{\xi_v^\vartheta = k\}} \mid \mathbb{T}_{\le r+1},\,  \mathbb{F}_{\le r}\big] \\&= \sum_{w\in \mathcal{G}_{r-1}} \sum_{v \in c(w)} \mathbf{1}_{\{\mathcal{F}_v\ge f\} }  \mathbb{E}\big[\mathbf{1}_{\{\xi^\vartheta_v = k\}} \mid \xi_v   \big].
    \end{split}\]
Now, using the tower property yields
       \[\begin{split}
        \mathbb{E}\big[Z_r \mid \mathbb{T}_{\le r},\,  \mathbb{F}_{\le r-1}\big] &= \mathbb{E}\Big[\mathbb{E}\big[Z_r \mid \mathbb{T}_{\le r+1},\,  \mathbb{F}_{\le r}\big]\mid \mathbb{T}_{\le r},\,  \mathbb{F}_{\le r-1} \Big] \\ &=   \sum_{w\in \mathcal{G}_{r-1}} \sum_{v \in c(w)}  \mathbb{E}\left[ \mathbf{1}_{\{\mathcal{F}_v\ge f\} } \mathbb{E}[  \mathbf{1}_{\{\xi^\vartheta_v = k\}} \mid \xi_v]  \mid \mathbb{T}_{\le r},\,  \mathbb{F}_{\le r-1} ] \right] \\ &=   \sum_{w\in \mathcal{G}_{r-1}} \xi_w \mathbb{P}(\xi^\vartheta= k, \, \mathcal{F} \ge f).
    \end{split}\]
Thus, once again appealing to the tower property we get 
         \[\begin{split}
        \mathbb{E}\big[Z_r \mid \mathbb{T}_{\le r-1},\,  \mathbb{F}_{\le r-2}\big] &=   \sum_{u\in \mathcal{G}_{r-2}} \sum_{w \in c(u)}  \mathbb{E}[\xi \mathbf{1}_{\{\mathcal{F}\ge \vartheta\}}]\mathbb{P}(\xi^\vartheta= k, \, \mathcal{F} \ge f)\\ &=  \sum_{u\in \mathcal{G}_{r-2}} \xi_u\mathbb{E}\big[\xi  \mathbf{1}_{\{\mathcal{F}\ge \vartheta\}}\big]\mathbb{P}( \xi^\vartheta= k, \, \mathcal{F} \ge f).
    \end{split}\]
    Recursively we deduce
         \[\begin{split}
        \mathbb{E}\big[Z_r \mid \mathbb{T}_{\le 2},\,  \mathbb{F}_{\le 1}\big] &=  \sum_{u\in \mathcal{G}_{1}} \xi_u \big(\mathbb{E}\big[\xi  \mathbf{1}_{\{\mathcal{F}\ge \vartheta\}}\big]\big)^{r-2}\mathbb{P}(\xi^\vartheta= k, \, \mathcal{F} \ge f).
    \end{split}\]
    Therefore, 
      $$\mathbb{E}[Z_r]= \mu_\vartheta^{r-2}\mathbb{P}(\xi^\vartheta= k, \, \mathcal{F} \ge f) \mathbb{E}\left[\sum_{u\in \mathcal{G}_{1}} \xi_u \right] \ge  k \mu_\vartheta^{r-2}\mathbb{P}(\xi^\vartheta= k, \, \mathcal{F} \ge f),$$
      where in the last inequality we have used that we are working on the probability measure conditional on the event $A_{f,k}^{\rho}$ which states that the set of children of $\rho$ with  $\xi_v\ge 1$ and $\mathcal{F}_v\ge \vartheta$ has cardinality $k$.
\end{proof}

Now we need the following two lemmas, which the reader
can find in \cite[Lemma 2.3 and Lemma 3.4]{pemantle1992contact}. The first lemma gives a
lower bound for the probability that a binomial random variable is at least~$1$. The second lemma gives a necessary condition for  $\liminf_{t\to\infty}g(t)$ to be positive, where $g$ is a function on the non-negative \nat{real} numbers.  

\begin{lemma}[{\cite[Lemma 2.3]{pemantle1992contact}}]\label{lem_pematle_bin}
	Let $M$ be a positive integer-valued random variable and pick $p<\mathbb{E} [M]$. For any $x\in (0,1]$, let $M_x$ be a random variable with binomial distribution \normalfont{\textbf{Bin}}$(M,x)$. Then there exits $\delta >0$ such that $$\mathbb{P}(M_x\geq 1)\geq  px\wedge \delta.$$ 
\end{lemma}

\begin{lemma}[{\cite[Lemma 3.4]{pemantle1992contact}}]\label{lem_pemantle}
	Let $G$ be any non-decreasing function on $[0,\infty)$ such that $G(x)\geq x$ on some neighbourhood of 0. Suppose $g$ is a function on $[0,\infty)$  that satisfies, for some $S>0$, 
	\begin{equation}\label{eq_condLemmaPemantle}	
		\inf_{0\leq t \leq S} g(t) >0\quad \quad \text{and}\quad  \quad g(t) \geq G\left( \inf_{0 \leq s \leq t-S} g(s) \right)\quad \quad   \text{for} \quad  t>S.
	\end{equation}
	Then, $$\liminf_{t \to \infty} g(t) >0.$$
\end{lemma}

To prove  Theorem \ref{teo_noexptail} we will first aim  to show that for any $\lambda>0$ we have $\liminf_{t \to \infty}\mathbb{P}(\rho \in X_t)>0$ where we average over the tree $(\mathcal{T},\mathbb{F}(\mathcal{T}))$. 
As we will explain below, this immediately implies that $\lambda_2 = 0$.

Now, our strategy is to apply Lemma \ref{lem_pemantle} to the function 
$$g(t):=\mathbb{P}(\rho\in X_t), \qquad t> 0.$$
The next step is to find  a non-decreasing function $G$ which satisfies \eqref{eq_condLemmaPemantle}. In the following lemma we do so under a technical condition that will be implied by Assumption \eqref{eq:offspringfitness} as we will see later on. 

\begin{lemma}\label{lem:boundG}
{\color{black}For $\epsilon = \epszero$,} suppose $S$, $ \widehat{C}_{\lambda, f}$, and $\widehat{\lambda}$ are as defined in   \eqref{eq_timeS} and \eqref{eq_lhat}, respectively. 
   Assume that there exists $r = r(f,k)\ge 2$ such that
    \begin{equation}\label{eq_condS}
		\frac{S}{4(2r+1)} -1> \frac{2}{\widehat{C}_{\lambda, f} \widehat{\lambda}^r}.
	\end{equation}
 	For either $f$ or $k$ large enough there exists a constant $c\in(0,1)$ and $\delta>0$ such that 
if we define
\begin{equation}\label{eq:fnG}
    G(x):= c k \mu^{r-2}_{\vartheta}\mathbb{P}\left(\xi^\vartheta=k, \, \mathcal{F}\geq f \mid A_{f,k}^{\rho}\right)x \wedge \delta, \quad \quad  \text{for}\quad x\geq 0,
\end{equation}
then it holds that
  	\begin{equation*}
		g(t)\ge G\left(\inf_{0\leq s\leq t-S} g(s)\right), \quad \text{for}\quad t>S.
	\end{equation*}
\end{lemma}

\begin{proof}
	Throughout this proof fix $\lambda > 0$.  
We will from now on assume that $\mathbb{P}$ refers to the conditional probability  measure given $A_{f,k}^{\rho}$.  Also, we emphasize that we always start with the root initially infected (unless specified otherwise).
	We first prove that, conditionally   on $v\in V_r$ being a star,  then it will be  infected before time $S$ with probability bounded away from zero uniformly for either $f$ or $k$ large enough. 
	That is to say, there exists a positive constant $c_1$ such that for either $f$ or $k$ sufficiently large and any $v \in  \nat{\mathcal{A}_{f,k}^r}$,
		\begin{equation*}
			p_{in}:=\mathbb{P}_{\mathcal{T}, \mathbb{F}}\big(v \in X_s \ \text{for some}\  s\in [0,S]\  \big)\geq c_1.
		\end{equation*}
	The proof of this bound follows by an application of  Lemma~\ref{lem_starchain}, 
 by using the montonicity of the inhomogeneous contact process restricted
to the subgraph consisting of a star with root $\rho$ together with a  path $\mathcal{C}_r$  from one descendent of the root to $v$.
 Indeed, for either $f$ or $k$ sufficiently large we obtain 
	\begin{equation*}
		1-p_{in}\le \big(1- \widehat{C}_{\lambda,f}\widehat{\lambda}^r\big)^{S/4(2r+1)-1} + R(f,k,\lambda),
	\end{equation*}
{where the function $R(f,k,\lambda)$ is defined in \eqref{eq_Rf}.}
	Then, using the inequality $(1-x)^{1/x}<e^{-1}$, observe that \eqref{eq_condS} forces the first term on the right-hand side of  the last equation to be at most $e^{-2}$, i.e.\ 
	\begin{eqnarray*}
		\big(1- \widehat{C}_{\lambda,f}\widehat{\lambda}^r\big)^{S/4(2r+1)-1} \leq \big(1- \widehat{C}_{\lambda,f}\widehat{\lambda}^r\big)^{2/(\widehat{C}_{\lambda, f} \widehat{\lambda}^r)}< e^{-2}.
	\end{eqnarray*}
	Further, for either $f$ or $k$ sufficiently large we have $R(f,k,\lambda)\leq 1/2$. Therefore, we deduce that $p_{in}$ is bounded away from 0 for either $f$ or $k$ large enough, i.e.
	\begin{equation}\label{eq_pnozero}
		p_{in} \geq c_1,\quad \quad  \text{where} \quad \quad c_1=\frac{1}{2}-e^{-2}>0.
	\end{equation}
	In other words,  with positive probability we push the infection to a  generation $r$ which satisfies condition \eqref{eq_condS}. 
	
Now conditioning on the event that $v\in X_{t-\nat{S}}$ for some $v\in \nat{\mathcal{A}_{f,k}^r}$, we deduce the following lower bound, for  $t>S$
	\begin{equation}\label{eq_gA}
		g(t) =   \mathbb{P}\big(\rho\in X_t  \big) \geq H_1(t) H_2(t),
	\end{equation}
	where the functions $H_1$ and $H_2$ are given by 
	$$H_1(t)= \mathbb{P}\big(v\in X_{t-S}\ \text{for some}\ v\in \nat{\mathcal{A}_{f,k}^r}\big),$$ 
	$$ H_2(t)= \mathbb{P}\big(\rho \in X_t \ \big|\big.  \  v\in X_{t-S}\ \text{for some}\ v\in \nat{\mathcal{A}_{f,k}^r}\big).$$ 
	Hence, the next goal is to  establish lower bounds for  the functions $H_1$ and $H_2$.

	\textbf{Lower bound for $H_1$.}  Denote by $Z_r= |\nat{\mathcal{A}_{f,k}^r}|$ the number of \nat{$k$-}stars vertices in generation~$r$. Taking into account that we are working conditionally on the event $A_{f,k}^{\rho}$, we have from Lemma \ref{lem:zr} that 
	\begin{equation*}
		\mathbb{E}[Z_r] \geq
		k \mu^{r-2}_{\vartheta}\mathbb{P}\left(\xi^\vartheta=k,\ \mathcal{F}\geq f\right).
	\end{equation*}
 Let $M_r^S$ be the random number of \nat{$k$-}stars in generation $r$ that are infected before time $S$. Together with \eqref{eq_pnozero}, the above estimate is sufficient to obtain that, 
	\begin{eqnarray}
		\mathbb{E} \big[M_r^S\big]&\geq& 
					\mathbb{E}\Bigg[ \sum_{v \in \nat{\mathcal{A}_{f,k}^r}} \mathbb{P}_{\mathcal{T}, \mathbb{F}} \big( v \in X_s \mbox{ for some } s \in [0,S]\big)  \Bigg]
		\\ & \geq& c_1 \E[ Z_r] \geq  c_1 k \mu^{r-2}_{\vartheta}\mathbb{P}\left(\xi^\vartheta=k, \,\mathcal{F}\geq f\right). \label{eq:mean_MrS}
	\end{eqnarray}	
	Let us define, for $t>2S$,
	$$\chi(t):= \inf\big\{g(s): 0\leq s \leq t-S\big\}.$$ 
	Now, ignore all the infections of $v\in \nat{\mathcal{A}_{f,k}^r}$ by its parent except the first infection. Then, the contact process on the subtrees rooted at vertices $v\in \nat{\mathcal{A}_{f,k}^r}$ which are infected at some time $s<S$ will evolve independently from time $s$ to time $t-S$ and then vertex $v$ will be infected with probability at least $\chi(t)$. 
	Therefore, if we
	denote by $M_r^{t-S}$  the random number of \nat{$k$-}stars in generation $r$ that are infected at time $t-S$, we can conclude that
	the random variable $M_r^{t-S}$  stochastically  dominates  a  random variable $M_\chi$ that has  distribution \textbf{Bin}$\big(M_r^S,\chi(t)\big)$. By Lemma \ref{lem_pematle_bin} and~\eqref{eq:mean_MrS}, there exists $\delta_1 >0$ such that
	\begin{eqnarray*}
		\mathbb{P}\big( M_\chi \geq 1\big) \geq
		{2^{-1}}c_1 k \mu^{r-2}_{\vartheta} \mathbb{P}\left(\xi^\vartheta=k, \,\mathcal{F}\geq f\right)\chi(t) \wedge \delta_1.
	\end{eqnarray*} 
{Note that the factor $2^{-1}$ is required to guarantee the hypotheses of Lemma~\ref{lem_pematle_bin}.}
		Moreover, since $M_r^{t-S}$  dominates the random variable $M_\chi$,
	we obtain that, for $t>2S$
	\begin{equation*}
		H_1(t)\geq  \mathbb{P}\big(M_r^{t-S}\geq1 \big) \geq \mathbb{P}\big( M_\chi \geq 1\big),
	\end{equation*} 
	which implies for $t>2S$
	\begin{equation}\label{eq_lowerboundH1}
		H_1(t)\geq {2^{-1}}c_1 k \mu^{r-2}_{\vartheta} \mathbb{P}\left(\xi^\vartheta =k, \, \mathcal{F}\geq f\right)\chi(t) \wedge \delta_1.
	\end{equation}

	\textbf{Lower bound for $H_2$.} Let $t>S$. We break down $H_2(t)$ into a product of conditional probabilities as follows 
	\begin{eqnarray*}
		H_2(t)= \mathbb{P}\big(\rho \in X_t \ \big|\big.  \  v\in X_{t-S}\ \text{for some}\ v\in \nat{\mathcal{A}_{f,k}^r}\big)\geq h(t)\widehat{h}(t), 
	\end{eqnarray*}
	where
	\[\begin{split}
		h(t)&= \mathbb{P}\big(\rho \in X_s \ \text{for some}\ s\in [t-S,t-fk^{2/3}-1]\ \big|\big. \  v \in X_{t-S}\ \text{for some}\ v\in \nat{\mathcal{A}_{f,k}^r}\big), \\ \widehat{h}(t)&= \mathbb{P}\big(\rho \in X_t \ \big|\big.\ \rho \in X_s \ \text{for some}\ s\in [t-S,t-fk^{2/3}-1],\  v\in X_{t-S}\ \text{for some}\ v\in \nat{\mathcal{A}_{f,k}^r}\big).
	\end{split}\]
	We can lower bound $h(t)$ by ignoring all the possible infections other than the infection of $v$ at time $t-S$, we have
	\begin{eqnarray*}
		h(t) \geq \mathbb{P}_{\{v\}}\big(\rho \in X_s \ \text{for some}\ s\leq S-fk^{2/3}-1 \ \big|\big. \ v \in \nat{\mathcal{A}_{f,k}^r}\big),
	\end{eqnarray*}
	where $\mathbb{P}_{\{v\}}$ denote the law of the process with $v$ initially infected. 
	Moreover, by ignoring one child of $v$ and considering the star of size $k$ centered at $v$, 
		the same argument as before applies and we deduce that
		\begin{equation*}\label{eq_lowerboundp1}
			\begin{split}
			h(t) \geq  \mathbb{P}_{\{v\}}\big(\rho \in X_s \ \text{for some}\ s\leq S-fk^{2/3}-1 \ \big|\big. \ v \in \nat{\mathcal{A}_{f,k}^r}\big) \geq c_1,
				\end{split}
		\end{equation*}
		where the constant $c_1$ is the same as in \eqref{eq_pnozero} and comes from assumption \eqref{eq_condS}. 
	
	On the other hand, once again by monotonicity of the contact process we have, for $t>S$ 
\begin{eqnarray*}
		\widehat{h}(t)  \geq   \mathbb{P}\big(\rho \in X_t \ \big|\big.\ \rho \in X_s \ \text{for some}\ s\in [t-S,t-fk^{2/3}-1]\big).
	\end{eqnarray*}
	Let us now introduce the stopping time for $X$,  
	$$ \tau= \inf\{u\geq t-S:\  \rho \in X_u\}.$$ 
Recall $\epsilon = \epszero$. On the event $\{\tau\leq t-fk^{2/3}-1\}$, let $\mathcal{B}_\tau$ be the event that the number of infected neighbours of $\rho$ is at least $\epsilon L$ in the entire random time interval 
$[\tau+fk^{2/3}, t-1]$, i.e.\ 
	\begin{equation*}
		\mathcal{B}_\tau=   \left\{ \inf_{u \in [\tau+fk^{2/3}, t-1]} | \Lambda_u| \geq \epsilon L \right\},
	\end{equation*}
	where we recall that $\Lambda_u$ is the set of infected leaves of $\rho$ at time $u$. Denote by $\mathscr{F}_t$  the $\sigma$-algebra generated by tree, fitness and the contact process up to time $t$. Then, appealing to the strong Markov property and Lemma \ref{lem_infecleaves} and using that $\tau \geq t-S$, we obtain
	\[\begin{split}
		\mathbb{P}(\mathcal{B}_\tau | \mathscr{F}_\tau) \geq \mathbb{P}\left(\inf_{u \in [\tau +fk^{2/3}, S+\tau]} |\Lambda_t| \geq \epsilon L \ \Big|\Big. \mathscr{F}_\tau \right) \geq 1- R(f,k,\lambda),
	\end{split}\]
	where the function $R(f,k,\lambda)$ is defined in \eqref{eq_Rf}. Next, define $\mathcal{I}$ to be the event that at least one of the $\epsilon L$ neighbours that is infected at time $t-1$ infects the root  at a time in $[t-1,t]$ before recovering.  
 
Denote by $x_1, \dots, x_{\lceil \epsilon L\rceil}$ the infected neighbours at time $t-1$. We can estimate the probability of $\mathcal{I}$ as
	\begin{equation*}
	\begin{split}
		\mathbb{P}(\mathcal{I} \mid \mathcal{F}_{t-1}) &\geq 1- \prod_{i=1}^{\lceil \epsilon L\rceil} \mathbb{P}(\{x_i\ \text{recovers}\  \text{in}\ [t-1,t]\}\cup \{x_i\ \text{does not infect}\ \rho \ \text{in}\ [t-1,t]\})\\ & = 1- \prod_{i=1}^{\lceil \epsilon L \rceil}(1- \mathbb{P}(\{x_i\ \text{does not recover}\  \text{in}\ [t-1,t]\}\cap \{x_i\ \text{infects}\ \rho \ \text{in}\ [t-1,t]\})\\ & \ge 1 - \prod_{i=1}^{\lceil \epsilon L \rceil} (1- e^{-1}(1-e^{-\lambda f \vartheta})) \geq 1- (1-e^{-1}(1-e^{-\lambda f \vartheta}))^{\epsilon L}=: 1 - a(f,k, \lambda).
	\end{split}
\end{equation*}
	Also, note that $\mathbb{P}(\mathcal{R}_\rho | \ \mathscr{F}_{t-1})\geq e^{-1}$, where $\mathcal{R}_{\rho}=\{\rho \ \text{does not recover in}\ [t-1,t]\}.$ With this notation  we have the following estimate
\[\begin{split}
		\widehat{h}(t) \geq \mathbb{P}\big(\mathcal{I}\cap \mathcal{R}_{\rho}\cap \mathcal{B}_\tau \mid \tau \leq t-fk^{2/3}-1\big).
	\end{split}\]
	Conditioning on $\mathscr{F}_{t-1}$ and using the independence of the infection and recovery events, we obtain that
	\[\begin{split}
		\mathbb{P}\big(\mathcal{I}\cap \mathcal{R}_{\rho}\cap \mathcal{B}_\tau\cap \{\tau \leq t-fk^{2/3}-1\}\big) &=\mathbb{E}\big[\mathbb{P}(\mathcal{I}\cap \mathcal{R}_\rho | \mathscr{F}_{t-1})\mathbf{1}_{\{\mathcal{B}_\tau, \tau\leq t-fk^{2/3}-1\}} \big] \\ & = \mathbb{E}\Big[\mathbb{P}(\mathcal{I}|\mathscr{F}_{t-1})\mathbb{P}(\mathcal{R}_\rho|\mathscr{F}_{t-1})\mathbf{1}_{\{\mathcal{B}_\tau, \tau\leq t-fk^{2/3}-1}\} \Big] \\ &\geq \left(1 - a(f,k, \lambda) \right)e^{-1}\mathbb{E}\Big[\mathbb{P}(\mathcal{B}_\tau |\mathscr{F}_{\tau})\mathbf{1}_{\{\tau\leq t-fk^{2/3}-1}\}\Big].
	\end{split}\]
 
	Combining with the previous estimates, we get the lower bound
	\[\begin{split}
		\widehat{h}(t)  \geq e^{-1}  \left(1 - a(f,k,\lambda) \right)\big(1-R(f,k, \lambda)\big).
	\end{split}\]
	Denote by $c_2$ the expression 
  $$c_2=c_2(f,k, \lambda):=e^{-1} \left(1 - a(f,k, \lambda) \right) \big(1-R(f,k,\lambda)\big).$$ 
	Observe that this expression  depends on $\lambda, f$ and $k$, however we see that $c_2\to e^{-1}$ as $k\to \infty$ and $f$ is fixed and  $c_2\to e^{-1}(1-e^{-1})^{\epsilon\lceil k/2\rceil}$ as $f\to \infty$ and $k\geq \kzero$ fixed. In other words, in our  two relevant cases we can lower bound $c_2$ by a constant.
	The above lower bounds for $h$ and $\widehat{h}$ give, for $t>S$
	\begin{equation}\label{eq_lowerboundH2}
		H_2(t)\geq h(t)\widehat{h}(t)  \geq c_1 c_2.
	\end{equation}
	Plugging  \eqref{eq_lowerboundH1} and \eqref{eq_lowerboundH2} back into \eqref{eq_gA}, we now see that, for $t>2S$
	\begin{eqnarray*}
		g(t)\geq H_1(t) H_2(t) \geq c k \mu^{r-2}_\vartheta \mathbb{P}\left(\xi^\vartheta=k, \, \mathcal{F}\geq f\right)\chi(t) \wedge \delta_1,
	\end{eqnarray*}
	where $c=2^{-1}c_1^2c_2 \in (0,1)$.  Hence, under the assumption that $r$ satisfies  condition \eqref{eq_condS}, we  establish the following lower bound
	\begin{equation*}
		g(t)\geq \left\{ \begin{array}{lcc}
			c k \mu^{r-2}_{\vartheta}\mathbb{P}\left(\xi^\vartheta=k, \, \mathcal{F}\geq f\right)\chi(t) \wedge \delta_1,&     & t>2S \\
			\\ \displaystyle \inf_{0 \leq s \leq 2S} g(s), &  & S \leq t \leq 2S.
		\end{array}
		\right.
	\end{equation*}
	Furthermore, note that 
	\begin{eqnarray*}
		\inf_{0 \leq s \leq 2S} g(s)  \geq \mathbb{P}\big(\rho \ \text{never recovers in} \ [0,2S]\big) = e^{-2S} >0.
	\end{eqnarray*}
	Therefore, the above two estimate are sufficient to deduce  that  there exists $\delta >0$ such that
	$$ g(t) \geq 	c k \mu^{r-2}_{\vartheta} \mathbb{P}\left(\xi ^\vartheta=k, \, \mathcal{F}\geq f\right)\chi(t) \wedge \delta , \quad \quad \text{for}\quad  t>S,$$
 thus completing the proof.
\end{proof}

The following final preparatory lemma gives us an equivalent formulation of hypothesis \eqref{eq:offspringfitness} which is easier to use in our final proof. Further, in the case of unbounded fitness we prove that \eqref{eq:offspringfitness} implies the analogue condition for $\xi^\vartheta$ defined in \eqref{eq:xivartheta} which is also need it in the last proof.

\begin{lemma}\label{lem:equivcond}
   Condition \eqref{eq:offspringfitness} is equivalent to 
 \begin{equation}\label{eq:mixingmoments}
  \limsup_{\substack{f+k\to \infty\\ \, {\color{black}k\ge \kzero}, \,  f\ge \vartheta}} \frac{\log \big(\mathbb{P}(\xi=k, \, \mathcal{F}  \geq f)\big)}{k\log (1
     +f)}= 0.
 		\end{equation}
   Furthermore,  
   Condition \eqref{eq:offspringfitness} implies \eqref{eq:mixingmoments} but with $\xi^\vartheta$ instead of $\xi$, where $\xi^\vartheta$ is defined in~\eqref{eq:xivartheta}.
\end{lemma}

 \begin{proof}
Let {\color{black}$k\ge \kzero$} and $f\ge \vartheta$. First, we assume that  \eqref{eq:offspringfitness} does not hold, i.e. 
 \begin{equation}\label{eq:mixingmomentsop}
 \mathbb{E}\big[(1+\mathcal{F})^{c\xi}\mathbf{1}_{\{{\color{black}\xi\ge \kzero}, \, \mathcal{F}\ge \vartheta\}}\big] < \infty, \qquad \text{for some}\quad  c>0.
 		\end{equation}
 Now, we note that 
 \[\begin{split}
 \mathbb{E}\big[(1+\mathcal{F})^{c\xi}\mathbf{1}_{\{ {\color{black}\xi\ge \kzero}, \, \mathcal{F}\ge \vartheta\}}\big] \ge \mathbb{E}\big[(1+\mathcal{F})^{c\xi}\mathbf{1}_{\{\xi=k, \, \mathcal{F}\ge f\}} \big] \ge (1+f)^{ck} \mathbb{P}(\xi=k,\,  \mathcal{F}  \geq f),
     \end{split}\]
 which implies 
 \[\begin{split}
 		\limsup_{\substack{f+k\to \infty\\ {\color{black}k\ge \kzero}, \, f\ge \vartheta}} \frac{\log \big(\mathbb{P}(\xi=k, \, \mathcal{F}  \geq f)\big)}{k\log (1+f)} \leq \limsup_{\substack{f+k\to \infty\\ {\color{black}k\ge \kzero}, \, f\ge \vartheta}} \frac{\log( (1+f)^{-ck} \mathbb{E}[(1+\mathcal{F})^{c\xi}\mathbf{1}_{\{ {\color{black}\xi\ge \kzero}, \, \mathcal{F}\ge \vartheta\}}])}{k \log (1+f)}=-c. 
 \end{split}\]
 In other words, we have showed that 
   \begin{equation}\label{eq:opplimsup}
       	\limsup_{\substack{f+k\to \infty\\ \, {\color{black}k\ge \kzero}, \,  f\ge \vartheta}} \frac{\log \big(\mathbb{P}(\xi=k, \, \mathcal{F}  \geq f)\big)}{k\log (1
     +f)}= -\delta,   \qquad \text{for some}\quad  \delta\in (0,\infty].
   \end{equation}
Now, assume that \eqref{eq:opplimsup} holds. Then there exist $\delta>0$ and $C>0$ such that 
 	 \begin{equation}\label{eq:inequality2}
\mathbb{P}(\xi=k, \, \mathcal{F}  \geq f) \leq C(1+f)^{-\delta k},\qquad \text{for all} \quad f\ge \vartheta \quad \text{and}\quad {\color{black}k\ge \kzero}.
 	 \end{equation}
  We choose $c\in (0,\delta)$ small enough such that $2^c(1+\vartheta)^{c-\delta}<1$. Now for some $k_0$ such that 
  $(\delta-c)k_0>1$, we observe that the following inequalities hold
 		\[	\begin{split}
 \mathbb{E}\big[(1+\mathcal{F})^{c\xi}\mathbf{1}_{\{\xi\geq k_0, \, \mathcal{F}\geq \vartheta\}}\big]  &\leq  \sum_{k=k_0}^{\infty} \sum_{f= 0}^\infty \mathbb{E}\big[(1+\mathcal{F})^{c\xi}\mathbf{1}_{\{\mathcal{F}\in [f+\vartheta, f+\vartheta+1)\}\cap \{\xi= k\}}\big] \\ & \leq   \sum_{k=k_0}^{\infty} \sum_{f= 0}^\infty (2+\vartheta+f)^{ck}\mathbb{P}(\xi =k, \, \mathcal{F}\geq f+\vartheta)\\ & \leq   C\sum_{k=k_0}^{\infty} \sum_{f= 0}^\infty \big(2^c(1+\vartheta+f)^{-(\delta-c)}\big)^k \\ & =  C\sum_{f= 0}^\infty  \frac{2^{ck_0}(1+\vartheta+f)^{-(\delta-c) k_0}}{1-2^c(1+\vartheta+f)^{-(\delta-c)}}\\ & \le \frac{C\, 2^{ck_0}}{1-2^c(1+\vartheta)^{c-\delta}}\sum_{f= 0}^\infty (1+\vartheta+f)^{- (\delta-c) k_0} < \infty,
 	\end{split}\]
  where in the third inequality we used that $(2+\vartheta+f)^{ck}\le 2^{ck} (1+\vartheta+f)^{ck}$ and in the equality we get a geometric sum which converges since $2^c(1+\vartheta+f)^{c-\delta}\le 2^c(1+\vartheta)^{c-\delta}$ for all $f\ge 0$.
On the other hand, 
	\[	\begin{split}
 \mathbb{E}\big[(1+\mathcal{F})^{c\xi}\mathbf{1}_{\{{\color{black}\kzero\le \xi}\leq k_0, \, \mathcal{F}\geq \vartheta\}}\big]  &\leq  \sum_{k=\kzero}^{k_0} \int_0^\infty \mathbb{P}((1+\mathcal{F})^{ck}\ge x, \, \xi =k, \, \mathcal{F}\ge \vartheta) \,\mathrm{d} x \\ & = \sum_{k=\kzero}^{k_0}\int_{0}^{(1+\vartheta)^{ck}} \mathbb{P}(\xi =k, \, \mathcal{F}\ge \vartheta) \,\mathrm{d}x  
 \\ & \quad + \sum_{k=\kzero}^{k_0}\int_{(1+\vartheta)^{ck}}^\infty \mathbb{P}(\xi =k, \, \mathcal{F}\ge x^{1/ck}-1) \, \mathrm{d}x \\
 & \le k_0 (1+\vartheta)^{ck_0} + C\sum_{k=\kzero}^{k_0}\int_{1}^\infty x^{-\delta/c} \,\mathrm{d}x <\infty,
 	\end{split}\]
  where in the last inequality we have used \eqref{eq:inequality2}. Thus we  get that \eqref{eq:mixingmomentsop} holds.
  
 For the second claim, first suppose that the $\limsup$ in~\eqref{eq:mixingmoments}
 is achieved along a sequence $(f_n,k_n)$ such that $f_n \rightarrow \infty$. Then, we see that 
\[\begin{split}
	\mathbb{P}(\xi^\vartheta = k_n,\  & \mathcal{F}\ge f_n) = \sum_{\ell=1}^{\infty} \mathbb{P}\bigg( \sum_{i=1}^{\ell} I_i = k_n  \bigg) \mathbb{P}(\xi =\ell, \,  \mathcal{F}\ge f)\\ &\geq \mathbb{P}\bigg(\sum_{i=1}^{k_n} I_i = k_n\bigg) \mathbb{P}(\xi= k_n, \,  \mathcal{F}\ge f_n) = \mathbb{P}(\mathcal{F}\geq \vartheta)^{k_n} \mathbb{P}(\xi =k_n, \, \mathcal{F}\ge f_n).
\end{split}\]
Hence,
\[\begin{split}
	\limsup_{\substack{f+k\to \infty\\ \, {\color{black}k\ge \kzero}, \,  f\ge \vartheta}}	\frac{\log \big(\mathbb{P}(\xi^{\vartheta}=k, \, \mathcal{F}  \geq f)\big)}{k\log (1+f)} \geq 	\lim_{n \rightarrow \infty} \frac{\log \big(\mathbb{P}(\mathcal{F}\geq \vartheta)^{k_n}\mathbb{P}(\xi=k_n, \, \mathcal{F}\geq f_n)\big)}{k_n\log (1+f_n)}=0,
\end{split}
\]
where in the last equality we have that $f_n \rightarrow \infty$. 

Now, suppose that the $\limsup$ in~\eqref{eq:mixingmoments} is achieved along
a sequence $(f_n,k_n)$ such that $k_n \rightarrow \infty$. Set $p = \mathbb{P}(\mathcal{F} \geq \vartheta)$. Then, we define $k_n^* \in \N$ such that $pk_n \leq k_n^* < p k_n + 1$, so that $\lfloor k_n^*/p \rfloor = k_n$. 
Similarly to before, we estimate
\[\begin{split}
	\mathbb{P}(\xi^\vartheta = k_n^*,  \mathcal{F}\ge f_n) 
 &\geq \mathbb{P}\bigg(\sum_{i=1}^{k_n} I_i = k_n^*\bigg) \mathbb{P}(\xi= k_n, \,  \mathcal{F}\ge f_n) .
\end{split}\]
Now, since $I_i$ are independent Bernoulli$(p)$ random variables,  by the de Moivre-Laplace theorem, we have that for $k_n$ large, there exists a $c > 0$ such that
\[ \mathbb{P}\bigg(\sum_{i=1}^{k_n} I_i = k_n^*\bigg) =
 \frac{(1+o(1))}{\sqrt{p(1-p)k_n}} \exp\Big\{- (k_n^* - p k_n)^2 / \sqrt{k_np(1-p)}\Big\} 
 \geq \frac{c}{\sqrt{k_n}} . \]
Thus, combining these estimates, we get that
\[ \frac{\log \big(\mathbb{P}(\xi^{\vartheta}=k_n^*, \, \mathcal{F}  \geq f_n)\big)}{k_n\log (1+f_n)} 
\geq \frac{\log( c/\sqrt{k_n})}{k_n \log (1+f_n)} + \frac{\log \big(\mathbb{P}(\xi=k_n, \, \mathcal{F}  \geq f_n)\big)}{k_n\log (1+f_n)} \rightarrow 0,
\]
as $n\to \infty$, which completes the proof.
 \end{proof}

With the previous preparatory lemmas in hand, we are now ready to prove Theorem \ref{teo_noexptail}.

\begin{proof}[Proof of Theorem \ref{teo_noexptail}]
 Throughout this proof fix $\lambda > 0$. 
We assume that  $\mathbb{P}$ refers to the conditional probability  measure given $A_{f,k}^{\rho}$ 
and as the latter event has positive probability it suffices to show $\liminf_{t \rightarrow \infty} \mathbb{P}(\rho \in \xi_t)> 0$ under the conditional probability measure.
 
	We would like to use Lemma \ref{lem_pemantle} applied to the non-negative and non-decreasing function $G$ defined in \eqref{eq:fnG}. 
	First, observe that 
	\begin{eqnarray*}
		\inf_{0 \leq s \leq S} g(s)  \geq \mathbb{P}\big(\rho \ \text{never recovers in} \ [0,S]\big) = e^{-S} >0.
	\end{eqnarray*}
Thus, we have that $g$ satisfies the first condition of Lemma \ref{lem_pemantle}. Now we want to apply Lemma \ref{lem:boundG} with the following choice of $r$, 
	\begin{equation}\label{eq_r_def}
		r ={r(f,k)=} \bigg\lceil -\frac{\log\big(\mu^{-2}_{\vartheta} c k\mathbb{P}(\xi^\vartheta =k, \, \mathcal{F}\geq f)\big) }{\log \mu_\vartheta} \bigg\rceil.
	\end{equation}
 We can see that as $\E[\xi^\vartheta]<\infty$ we have 
 by dominated convergence that 
\[ k \mathbb{P}(\xi^\vartheta = k, \mathcal{F} \geq f) 
\leq \E[ \xi^\vartheta \mathbf{1}_{ \{  \xi^\vartheta \geq k , \mathcal{F} \geq f \} } ] \rightarrow 0, 
\]
as either $k \to \infty$ or $f \to \infty$. Therefore, in either case $r = r(f,k)\rightarrow \infty$. Moreover, we have $c k \mu^{r-2}_{\vartheta} \mathbb{P}\left(\xi^\vartheta =k, \, \mathcal{F}\geq f\right)\geq 1$, and this in turn implies that $G(x)\geq x$ for some neighbourhood of 0. We complete our argument
		by showing that
		under the hypothesis of the theorem, we can find $f$ and $k$ such that~\eqref{eq_condS} holds.  This will allow us to apply Lemma~\ref{lem:boundG} together with Lemma~\ref{lem_pemantle} to \nat{obtain}
		\begin{equation*}
			\liminf_{t \to \infty} \mathbb{P}(\rho \in X_t)  >0 , 
		\end{equation*}
{\color{black}first when conditioning on $A_{f,k}^\rho$, but since this event has positive probability also without conditioning.}
        \nat{Now by the reverse Fatou lemma, we conclude  that
    \[ \mathbb{P}\big( \forall s \geq 0 \, \exists t \geq s \, : \, \rho \in X_{t} \big) = \mathbb{E}\Big[ \limsup_{t \rightarrow \infty} \mathbf{1}_{\{ \rho \in X_t \}} \Big] \geq
    \limsup_{t\to\infty} \mathbb{P}(\rho \in X_{t}) > 0,\]
which means that the process survives strongly.
    }
	
 It remains to show that~\eqref{eq_condS} holds.
 We begin by noting that condition \eqref{eq_condS} is implied by 
	\begin{equation}\label{eq:S2}
		\frac{S}{4(2r+1)} > \frac{4}{\widehat{C}_{\lambda, f} \widehat{\lambda}^r}.
	\end{equation}
Indeed, since $\widehat{\lambda}<1$, the right-hand side in \eqref{eq_condS} goes to infinity as $r\to \infty$ so that we can ignore the $-1$ on the left-hand side of the equation \eqref{eq_condS}. Now, since $C_{\lambda, f}\geq 1$, we see  
	from the definition of $S$ given in \eqref{eq_timeS} \nat{with $\epsilon=\epszero$}, that
	\begin{equation*}
		(1-e^{-\lambda \vartheta fL})(1-e^{-\gamma})\left(1+\frac{\lambda f  \vartheta}{2}\right)^{{L(1-2\epsilon)-1}} \left(\frac{\lambda   \vartheta^2}{\lambda   \vartheta^2 +1}\right)^r > 16k(2r +1),
	\end{equation*}
	implies condition \eqref{eq:S2}. 
    Recall that $\epsilon = \epszero$, so  the latter inequality is equivalent to 
	\begin{equation*}
		\log\left(\frac{\lambda   \vartheta^2}{1+\lambda   \vartheta^2}\right) > \frac{1}{r}  \left(\log\left(\frac{16k(2r+1)}{(1-e^{-\lambda \vartheta fL})(1-e^{-\gamma})}\right)-\left({\color{black}
        \frac{3L}{{8}}}-1\right)\log\left(1+\frac{\lambda f   \vartheta}{2}\right)\right).
	\end{equation*}
Taking into account that by the definition of $L$ given in \eqref{eq:defL}, $fL \rightarrow \infty $ as either $k \rightarrow \infty$ or $f\rightarrow \infty$, we have
			\[ \frac{1}{r}\log\left(\frac{16(2r+1)}{(1-e^{-\lambda \vartheta fL})(1-e^{-\gamma})}\right) \to 0,\quad  \text{as\,   either}\quad  k\to \infty\quad  \text{or}\quad  f\to \infty.\]
   It follows that \eqref{eq:S2} holds if  
   \begin{equation}\label{eq_ineq_r}
		\log\left(\frac{\lambda   \vartheta^2}{1+\lambda   \vartheta^2}\right) > \frac{1}{r}  \left(\log k-\left({\color{black}\frac{3L}{{8}}}-1\right)\log\left(1+\frac{\lambda f   \vartheta}{2}\right)\right) =: \Delta(f,k).
	\end{equation}
 As this condition is equivalent to 
 \[ \lambda >\frac{\exp \{ \Delta(f,k) \}}{\vartheta^2 ( 1- \exp \{ \Delta(f,k) \}) },  \]
 it suffices to show that $\Delta(f_n,k_n) \rightarrow - \infty$, for suitably chosen $f_n,k_n$, to complete the proof.

By Assumption \eqref{eq:offspringfitness} and Lemma \ref{lem:equivcond}, we have that there exists a sequence $(f_n, k_n)_{n\geq 1}$ such that $f_n + k_n \to \infty $ and
\begin{equation}\label{eq:reformulation} 	\lim_{n\to \infty}  \frac{\log \mathbb{P}(\xi^\vartheta=k_n, \, \mathcal{F} \geq  f_n)}{k_n\log (1+f_n)}=0.\end{equation}

{\color{black}Note that  either 
\begin{equation}\label{eq:bounds_L}   \begin{array}{ll} L \rightarrow \lceil k / 2  \rceil \geq 4 & \mbox{if } f_n \to \infty, \ k_n =  k \geq \kzero, \mbox{ or }\\
 L \geq \frac{\lambda \vartheta^2}{1+2\lambda \vartheta^2} k_n & \mbox{if } k_n \rightarrow \infty .\end{array} \end{equation}
Therefore,
\[ \log k_n / \big({\color{black}3 L /8} - 1)  \log (1 + f_n)\big) \to 0, \quad  \text{as\,   either}\quad  k_n\to \infty\quad  \text{or}\quad  f_n\to \infty. \]
{\color{black}Thus, we can deduce from the definition of $r$ in~\eqref{eq_r_def} that $\Delta(f_n,k_n) \to -\infty$
if }
\[ \frac{\big( {\color{black}3L/8} - 1\big)  \log (1 + f_n)}{\log \mathbb{P} (\xi^\vartheta = k_n,\,  \mathcal{F} \geq f_n )} \to -\infty,   \]
which now follows again by~\eqref{eq:bounds_L} directly from~\eqref{eq:reformulation}.}
This, shows that for this choice of $(f_n,k_n)$
we can guarantee that for $n$ large enough 
$\lambda > e^{ \Delta(f_n,k_n)}/(\vartheta^2 (1 - e^{ \Delta(f_n,k_n)}))$, so that condition~\eqref{eq:S2} holds 
and the process survives strongly.
\end{proof}

\textbf{Acknowledgments}:  The authors would like to thank the anonymous referee and the editor for their detailed comments and suggestions which helped us to improve this article.
N.C.-T. acknowledges support from CONACyT-MEXICO grant no. 636133. This manuscript was partially prepared while N.C.-T. was visiting the Department of Mathematical Sciences at the University of Bath, and she is grateful for the hospitality and collaboration. N.C.-T.  acknowledges support from SAMBa (Statistical Applied Mathematics at Bath) and Dorothea Schlözer-Programm at Georg-August-Universität Göttingen.

\bibliographystyle{abbrv}
\setlength{\bibsep}{1pt plus 0.3ex}
\bibliography{referencias}

\begin{thebibliography}{10}

\bibitem{zsolt2023}
Z.~Bartha, J.~Komjáthy, and D.~Valesin.
\newblock Degree-penalized contact processes.
\newblock {\em Forum of Mathematics, Sigma}, 14:e6, 2026.

\bibitem{berger2005spread}
N.~Berger, C.~Borgs, J.~T. Chayes, and A.~Saberi.
\newblock On the spread of viruses on the internet.
\newblock In {\em Proceedings of the {S}ixteenth {A}nnual {ACM}-{SIAM} {S}ymposium on {D}iscrete {A}lgorithms}, pages 301--310. ACM, New York, 2005.

\bibitem{bezuidenhout1990critical}
C.~Bezuidenhout and G.~Grimmett.
\newblock The critical contact process dies out.
\newblock {\em The Annals of Probability}, 18(4):1462--1482, 1990.

\bibitem{bhamidi2021survival}
S.~Bhamidi, D.~Nam, O.~Nguyen, and A.~Sly.
\newblock Survival and extinction of epidemics on random graphs with general degree.
\newblock {\em Ann. Probab.}, 49(1):244--286, 2021.

\bibitem{chatterjee2009contact}
S.~Chatterjee and R.~Durrett.
\newblock Contact processes on random graphs with power law degree distributions have critical value 0.
\newblock {\em Ann. Probab.}, 37(6):2332--2356, 2009.

\bibitem{Chung2002TheAD}
F.~Chung and L.~Lu.
\newblock The average distance in a random graph with given expected degrees.
\newblock {\em Internet Math.}, 1(1):91--113, 2003.

\bibitem{durret1995notes}
R.~Durrett.
\newblock Ten lectures on particle systems.
\newblock In {\em Lectures on probability theory ({S}aint-{F}lour, 1993)}, volume 1608 of {\em Lecture Notes in Math.}, pages 97--201. Springer, Berlin, 1995.

\bibitem{durret2019book}
R.~Durrett.
\newblock {\em Probability---theory and examples}, volume~49 of {\em Cambridge Series in Statistical and Probabilistic Mathematics}.
\newblock Cambridge University Press, Cambridge, 2019.

\bibitem{harris1974contact}
T.~E. Harris.
\newblock Contact interactions on a lattice.
\newblock {\em Ann. Probability}, 2:969--988, 1974.

\bibitem{huang2019exponential}
X.~Huang.
\newblock Exponential growth and continuous phase transitions for the contact process on trees.
\newblock {\em Electron. J. Probab.}, 25:Paper No. 77, 21, 2020.

\bibitem{huang2018contact}
X.~Huang and R.~Durrett.
\newblock The contact process on random graphs and {G}alton {W}atson trees.
\newblock {\em ALEA Lat. Am. J. Probab. Math. Stat.}, 17(1):159--182, 2020.

\bibitem{liggett1996binary}
T.~M. Liggett.
\newblock Multiple transition points for the contact process on the binary tree.
\newblock {\em Ann. Probab.}, 24(4):1675--1710, 1996.

\bibitem{liggett2013stochastic}
T.~M. Liggett.
\newblock {\em Stochastic interacting systems: contact, voter and exclusion processes}, volume 324 of {\em Grundlehren der Mathematischen Wissenschaften [Fundamental Principles of Mathematical Sciences]}.
\newblock Springer-Verlag, Berlin, 1999.

\bibitem{xue2017regular}
Y.~Pan, D.~Chen, and X.~Xue.
\newblock Contact process on regular tree with random vertex weights.
\newblock {\em Front. Math. China}, 12(5):1163--1181, 2017.

\bibitem{pastor2001}
R.~Pastor-Satorras and A.~Vespignani.
\newblock Epidemic dynamics and endemic states in complex networks.
\newblock {\em Phys. Rev. E}, 63:066117, May 2001.

\bibitem{pastor2001b}
R.~Pastor-Satorras and A.~Vespignani.
\newblock Epidemic spreading in scale-free networks.
\newblock {\em Phys. Rev. Lett.}, 86:3200--3203, Apr 2001.

\bibitem{pemantle1992contact}
R.~Pemantle.
\newblock The contact process on trees.
\newblock {\em Ann. Probab.}, 20(4):2089--2116, 1992.

\bibitem{peterson2011contact}
J.~Peterson.
\newblock The contact process on the complete graph with random vertex-dependent infection rates.
\newblock {\em Stochastic Process. Appl.}, 121(3):609--629, 2011.

\bibitem{stacey1996existence}
A.~M. Stacey.
\newblock The existence of an intermediate phase for the contact process on trees.
\newblock {\em Ann. Probab.}, 24(4):1711--1726, 1996.

\bibitem{valesin2024}
D.~Valesin.
\newblock The contact process on random graphs.
\newblock {\em Ensaios Mat}, 40:1--115, 2024.

\bibitem{xue2013regular}
X.~Xue.
\newblock Contact processes with random connection weights on regular graphs.
\newblock {\em Phys. A}, 392(20):4749--4759, 2013.

\bibitem{xue2015lattice}
X.~Xue.
\newblock Contact processes with random vertex weights on oriented lattices.
\newblock {\em ALEA Lat. Am. J. Probab. Math. Stat.}, 12(1):245--259, 2015.

\bibitem{xue2016recovery}
X.~Xue.
\newblock Critical value for the contact process with random recovery rates and edge weights on regular tree.
\newblock {\em Phys. A}, 462:793--806, 2016.

\bibitem{xue2019lattice}
X.~Xue.
\newblock The survival probability of the high-dimensional contact process with random vertex weights on the oriented lattice.
\newblock {\em ALEA Lat. Am. J. Probab. Math. Stat.}, 16(1):49--83, 2019.

\end{thebibliography}

\end{document}